\documentclass[reqno]{amsart}
\usepackage{latexsym,amsfonts,amssymb,amsmath,amscd,graphics}
\usepackage{amsthm}
\usepackage[mathscr]{eucal}
\usepackage{mathrsfs}
\usepackage{mathbbol}
\usepackage{oldgerm,units}

\def\ddispace{\setlength{\itemsep}{2pt}}
\def\htR{\widehat{R}}

\usepackage{ifthen}

\def\int{{\operatorname{int}}}

\def\Lz{L}

\def\Mon{\operatorname{Mon}}
\def\corn{{\operatorname{corn}}}

\def\tot{{\operatorname{comb}}}

\def\bidomains0{bi-\domains0}
\def\bisemifield0{bi-\semifield0}
\def\bisemifields0{bi-\semifields0}
\def\predomain0{pre-\domain0}
\def\predomains0{pre-\domains0}

\def\Cong{\Omega}

\def\cngA{\equiv}

\def\unscr{{\underline{\phantom{w}}}}

\def\SET{\operatorname{Set}}
\def\scrR{\mathscr R}

\def\hom{\operatorname{Hom}}
 \def\ext{\operatorname{Ext}}

\def\LAff{L\operatorname{-TropAff}}

\def\LCoord{L\operatorname{-Coord}}
\def\tlR{\widetilde R}
\def\tlPhi{\widetilde \Phi}

\def\htvrp{\widehat \vrp}

\def\R {R}

\def\mfC{\frak C}
\def\mfD{\frak D}

\def\otM{(\tM, \cdot \;, \leq)}

\def\gr{\operatorname{gr}}

\newcommand{\ds}[1]{\ {#1} \ }
\newcommand{\dss}[1]{\quad {#1} \quad }
\newcommand{\deqv}[1]{{#1}/{ \equiv}}
\newcommand{\deqvb}[1]{{#1}/{\Cong}}

\def\ldR{(\R, L, (\nu_{m,\ell} ))}

\def\ldsPR{(\R, L, s, P,(\nu_{m,\ell} ))}

\def\RLsnu{(\R, L, s, (\nu_{m,\ell}))}
\def\RLsnuT{(\R', L',  s', (\nu'_{m',\ell'}))}

\def\pipeWL{{\underset{L}{\mid}}}

\def\lmodWL{\mathrel  \pipeWL   \joinrel \joinrel =}

\def\httM{\overline{[\tM]}}

\newcommand\boxtext[1]{\pSkip \qquad \qquad \qquad \framebox{\parbox{\ltw}{#1}}\pSkip}
\def\ltw{0.7\textwidth}

\def\FR{R}
\def\vmap{\vartheta}

\newcommand{\xl}[2]{\,\,{^{[#2]}}{#1}\,}

\def\mfa{\frak a}

\def\mfp{\frak p}

\def\lone{1_L}

\def\tSS{\tS}
\def\mcA{\mathcal A}

\usepackage{epsfig,latexsym,amsfonts,amssymb,amsmath,amscd,graphics,epic}
\usepackage{amsfonts,amssymb,amsmath,amscd,amsthm}
\usepackage[mathscr]{eucal}
\usepackage{mathrsfs}
\usepackage{oldgerm,units}
\usepackage{wrapfig,epsfig}

\usepackage{ifthen}
\usepackage{mathbbol}

\usepackage{amsthm}
\usepackage[mathscr]{eucal}
\usepackage{mathrsfs}
\usepackage{mathbbol}
\usepackage{oldgerm,units}
\usepackage{wrapfig}

\newtheorem{theorem}{Theorem}[section]

\newtheorem{definition}[theorem]{Definition}

\newtheorem{example}[theorem]{Example}
\newtheorem{remark}[theorem]{Remark}

\newtheorem{note}[theorem]{Note}

\newcommand{\Real}{\mathbb R}
\newcommand{\Rati}{\mathbb Q}

\newcommand{\Net}{\mathbb N}

\newcommand{\Fld}{\mathbb K}

\newcommand{\one}{\mathbb{1}}
\newcommand{\zero}{\mathbb{0}}

\newcommand{\got}[1]{\frak{#1}}

\newcommand{\trop}[1]{\mathcal{#1}}

\newcommand{\tA}{\trop{A}}

\newcommand{\tD}{\trop{D}}

\newcommand{\tF}{\trop{F}}
\newcommand{\tG}{\trop{G}}
\newcommand{\tH}{\trop{H}}
\newcommand{\tI}{\trop{I}}

\newcommand{\tM}{\trop{M}}

\newcommand{\tS}{\trop{S}}

\newcommand{\tZ}{\trop{Z}}

\newcommand{\To}{\longrightarrow }

\newcommand{\Hom}{\operatorname{Hom}}

\newcommand{\al}{\alpha}
\newcommand{\bt}{\beta}
\newcommand{\gm}{\gamma}

\newcommand{\lm}{\lambda}
\newcommand{\Lm}{\Lambda}

\newcommand{\Val}{Val}
\newcommand{\nVal}{Val}

\newcommand{\id}{\inva{\aad}}

    \ifx\proof\undefined
    \newenvironment{proof}{
    \smallskip
    \noindent\emph{Proof.}}{\hfill\(\Box\)
    \bigskip
    } \fi

\newcommand{\bfem}[1]{\textbf{\emph{#1}}}

\newcommand{\ifdef}[3]{\ifthenelse{\equal{#1}{true}}{#2}{#3}}

\def\({\left(}
\def\){\right)}
 \input amssymb.sty
\input xy
\xyoption{all} \textwidth 160mm \textheight 228mm \topmargin -5mm
\def\Box{\operatorname{Box}}

\def\stak(f){\got{P}}
\def\stak{\got{S}}

\def\vrp{\varphi}

\def\htvrp{\widehat{\vrp}}
\def\tlvrp{\widetilde \varphi}
\def\tlphi{\widetilde \phi}

\def\tlk{\tilde k}
\def\tlell{\tilde \ell}
\def\tlv{\tilde v}

\def\tFlay{\mathcal F_{\operatorname{lay}}}

\def\nVal{\operatorname{Val}}

\def\lv{\operatorname{s}}

\def\nuge{\ge_{\nu}}

\def\nucong{\cong_{\nu}}

\def\C{\mathbb C}

\def\Fun{\operatorname{Fun}}
\def\Pol{\operatorname{Pol}}
\def\Lau{\operatorname{Laur}}

\def\maxplus{\mathcal M}

\def\({\left(}
\def\){\right)}
\def\Z{{\mathbb Z}}
\def\Q{{\mathbb Q}}

\def\semiring0{semiring$^\dagger$}
\def\semialg0{semi-algebra$^\dagger$}
\def\semirings0{semirings$^\dagger$}
\def\domain0{domain$^\dagger$}
\def\domains0{domains$^\dagger$}
\def\field0{semifield$^\dagger$}
\def\semifield0{semifield$^\dagger$}
\def\semifields0{semifields$^\dagger$}
\def\bidomain0{bi-\domain0}

\def\module0{module$^\dagger$}
\def\modules0{modules$^\dagger$}

\def\semifield0{semifield$^\dagger$}

\def\nucong{\cong_\nu}

\def\nuge{\ge_\nu}

\newcommand{\etype}[1]{\renewcommand{\labelenumi}{(#1{enumi})}}
\def\eroman{\etype{\roman}}
\def\ealph{\etype{\alph}}

\def\pipeL{{\underset{{L}}{\mid}}}
\def\lmodL{\mathrel  \pipeL \joinrel \joinrel =}

\def\ealph{\etype{\alph}}

\def\pSkip{\vskip 1.5mm \noindent}

\def\diag{\operatorname{diag}}
\def\a{\alpha}
\newtheorem{thm}[theorem]{Theorem}

\newtheorem*{thm*}{Theorem}

\newtheorem{cor}[theorem]{Corollary}

\def\Hom{\operatorname{Hom}}

\newtheorem{lem}[theorem]{Lemma}
\newtheorem{rem}[theorem]{Remark}
\newtheorem{prop*}{Proposition}

\newtheorem{prop}[theorem]{Proposition}
\newtheorem{defn}[theorem]{Definition}
\newtheorem{construction}[theorem]{Construction}

\newtheorem*{examp*}{Example}
\newtheorem*{examples*}{Examples}
\newtheorem*{remark*}{Remark}

\newtheorem{Note}[theorem]{Note}
\newtheorem*{defn*}{Definition}

\def\TropfunoneL{\tF_{\operatorname{LTrop}}}
\def\hTropfunoneL{\widehat{\tF}_{\operatorname{LTrop}}}

\def\Tropfunexp{\mathcal F_{\operatorname{LTrop;exp}}}
\def\Tropfuntwo{\mathcal F_{\operatorname{LTrop;unit}}}

\def\la{\lambda}

\def\Fun{\operatorname{Fun}}

\numberwithin{equation}{section}

\def\M0{M_{\zero}}

\def\id{\operatorname {id}}

\def\SR{R}

\def\SRo{(\SR, + \;,\; \cdot \;, \rone )}

\def\PS{P}

\def\rzero{\zero_\SR}

\def\rone{{\one_\SR}}
\def\mone{{\one_\tM}}

\def\rzero{\zero_\SR}

\def\cFWNV{\operatorname{ValField}}

\def\Valmon{\operatorname{ValMon}}
\def\SValmon{\operatorname{ValMon}^+}

\def\Valmonone{\operatorname{ValMon}_{(\one)}}

\def\Valfld{\operatorname{ValField}}
\def\Valdom{\operatorname{ValDom}}
\def\Mon{\operatorname{Mon}}
\def\Ring{\operatorname{Ring}}
\def\PrePOMon{\operatorname{PPreOMon}}
\def\PreOMon{\operatorname{PreOMon}}
\def\POMon{\operatorname{POMon}}
\def\SPOMon{\operatorname{POMon}^+}
\def\SOMon{\operatorname{OMon}^+}
\def\OMon{\operatorname{OMon}}

\def\SeR{\operatorname{Semir}}
\def\Bipot{\operatorname{Bipot}^\dagger}
\def\SeRo{\SeR^\dagger}

\def\LayDom{\operatorname{LayBidom}^\dag}

\def\QLayPreD{\operatorname{LayPreD}^\dag}

\def\ULayDom{\operatorname{ULayBidom}^\dag}
\def\LayDom{\operatorname{LayBidom}^\dag}

\newcommand{\nPS}[1]{\PS_{(!#1)}}
\newcommand{\nPSo}[1]{\nPS{\one}}

\def\srHom{\varphi}

\def\bfa{ \textbf{a}}

\def\bfi{ \textbf{i}}

\begin{document}


\title[Categorical layered mathematics]
{Categorical notions of \\[2mm] layered tropical algebra and geometry}

\author[Z. Izhakian]{Zur Izhakian}
\address{Department of Mathematics, Bar-Ilan University, Ramat-Gan 52900,
Israel}
\email{zzur@math.biu.ac.il}

\author[M. Knebusch]{Manfred Knebusch}
\address{Department of Mathematics, University of Regensburg, Regensburg,
Germany} \email{manfred.knebusch@mathematik.uni-regensburg.de}

\author[L. Rowen]{Louis Rowen}
\address{Department of Mathematics, Bar-Ilan University, Ramat-Gan 52900,
Israel} \email{rowen@macs.biu.ac.il}

\thanks{This research of the first and third authors is supported  by the
Israel Science Foundation (grant No.  448/09).}

\thanks{The research of the first author has been supported  by the
Oberwolfach Leibniz Fellows Programme (OWLF), Mathematisches
Forschungsinstitut Oberwolfach, Germany.}

\thanks{The second author was supported in part by the Gelbart Institute at
Bar-Ilan University, the Minerva Foundation at Tel-Aviv
University, the Mathematics Dept. of Bar-Ilan University, and the
Emmy Noether Institute. Research on this paper was carried out by the three authors in the Resarch
in Pairs program of the MFO in Oberwohlfach.}

\thanks{We thank Steve Shnider and Erez Sheiner  for explaining the proof of Payne's theorem in our tropical seminar.}

\date{\today}




\begin{abstract} This paper supplements
\cite{IzhakianKnebuschRowen2009Refined}, showing that
categorically the layered  theory is the same as the theory of
ordered monoids (e.g. the max-plus algebra) used in tropical
mathematics. A layered  theory is developed in the context of
categories, together with a ``tropicalization functor'' which
permits us to pass  from usual algebraic geometry to the tropical
world. We consider tropical varieties from this categorical
viewpoint, with emphasis on polynomial functions and their roots.
\end{abstract}

\maketitle




\section{Introduction}
\numberwithin{equation}{section}

Tropical geometry has led to considerable mathematical success in
degenerating various combinatoric questions. At the algebraic
level, the degeneration often has led to the max-plus algebra, but
in certain ways this is too crude a process to preserve many
important algebraic properties. Over the last few years, the
theory of supertropical algebras has been developed in a series of
papers including \cite{IzhakianKnebuschRowen2009Valuation},
\cite{IzhakianRowen2007SuperTropical},
\cite{IzhakianRowen2008Matrices},
\cite{IzhakianRowen2009Equations},  and
~\cite{IzhakianRowen2008Resultants},
 in which classical notions of
commutative algebra pertaining to algebraic varieties, matrices,
and valuations, carry over intrinsically to the ``tropical''
world. This degeneration still is too crude to handle several
issues, such as multiple roots of polynomials. A more refined
structure, called $L$-layered domains, was introduced in
\cite{IzhakianKnebuschRowen2009Refined} together with its basic
traits, in order to be able to preserve more algebraic properties.

Recent years have seen considerable progress in the algebraic
perspective of tropical geometry. Notably, building on work of
Bieri-Groves \cite{BG}, Berkovich \cite{Berk90} and Payne
\cite{Payne09} have shown how to view the analytification of an
affine variety algebraically, in terms of valuations and
multiplicative seminorms extending valuations.

This paper is part of a series including
\cite{IzhakianKnebuschRowen2009Refined} and \cite{IKR6}. In
\cite{IzhakianKnebuschRowen2009Refined} we  showed by example how
 the layered structure can cope with algebraic aspects
of tropical mathematics that are inaccessible to less refined
structures; some of these examples are reviewed here for the
reader's convenience.

Our main purpose in this paper
 is to provide a more formal, unified foundation
for further study, for both the algebraic and geometric aspects.
Since category theory pervades modern mathematics so thoroughly,
one feels obligated to describe the theory in categorical terms,
and indeed this language provides valuable information as to how
the theory should progress,  thereby throwing further light on
tropical geometry. We aim to understand those categories arising
from algebraic considerations,
 focusing on those algebraic aspects of the theory
that reflect most directly on tropical geometry, largely via a
Zariski-type correspondence. To describe these categories in full
detail would involve an inordinate amount of technical detail, so
we often make simplifying assumptions when they do not impact on
the tropical applications. Even so, each aspect of the theory
involves its corresponding categories, and so there are many
categories to be described here.  Although the language involves some technicalities,
we try to keep it to a minimum, leaving subtler matters to  \cite{IKR6}.
Another related paper is \cite{IKR5}, which delves into considerable detail in
the supertropical setting, for which we generalize parts to the layered setting,

To obtain the appropriate functors, we need first to make
categories of the classical ``algebraic world'' and the ``tropical
world.'' Informally, the classical ``algebraic world'' is
described by the categories associated to classical algebraic
geometry, often over the complex numbers $\mathbb C.$

 A deep connection between tropical geometry and valuation theory
 is already implicit in \cite {BG}, and
 it is convenient to work over algebraically closed fields with valuation.
 Thus, as the algebraic  aspect of the tropical theory has developed, $\mathbb C$ has been
replaced by the field of Puiseux series,  an algebraically closed
field endowed with a (nonarchimedean) valuation, whose target is
an ordered group, so it makes sense to work with  ordered groups,
or, slightly more generally,  ordered monoids.

There has been considerable recent interest in developing
algebraic geometry over arbitrary monoids \cite{Berk2011,CHWW},
and we shall draw on their work. One theme of this paper is  how
the assumption of an order on the monoid enriches the theory. The
ordered monoid most commonly used in the tropical world is the
``max-plus'' algebra $\maxplus$ (or its dual, the ``min-plus''
algebra), cf.~\cite{ABG}, \cite{AGG}, \cite{IMS}, and \cite{MS09}.
Our first main result (Proposition~\ref{maxplus1}), which sets the
flavor for the paper, is that the category of ordered cancellative
monoids is isomorphic to the category of bipotent semirings
(without a zero element). Any ordered monoid can be viewed as a
semiring, where multiplication is the given monoid operation and
addition is defined by taking $a+b$ to be  $\max \{ a, b\}$ in
Proposition \ref{maxplus1}. The universal of the appropriate
forgetful functor is constructed in this context, in Proposition
\ref{FrobUn}.

Since the underlying algebraic structures now are semirings, we
switch to the language of semirings in order to be able to adapt
concepts from ring theory and module theory, such as polynomials
and matrices.  We find it more convenient to work in the category
of \textbf{\semirings0}, defined as semirings not necessarily
having a zero element, for the following reasons:

\begin{itemize}
    \item The duality given in Proposition~\ref{rmk:duality} holds for
    \semirings0 but not for semirings; \pSkip
    \item Proofs are usually neater for \semirings0, since the
    zero
    element $\zero$ of a semiring needs special treatment; \pSkip
    \item Many important examples (such as Laurent series and tori) are defined over \semirings0 but
    not over semirings (and in particular, Bieri-Groves' main theorem \cite[Theorem A]{BG} is given
    for multiplicative groups); \pSkip
\item Once we get started with the layered theory, it is more
natural to utilize a 0-layer (an ideal comprised of several
elements) rather than a single   element $\zero$; anyway, one can
recover the element $\zero$ by inserting it into the 0-layer.
\end{itemize}

  One might counter that various critical aspects
of geometry such as intersection theory (which involve curves such
as $xy = \zero$) require a zero element. This turns out to be less
important in the tropical theory since the zero element,
$-\infty$, already is artificial, and can be dealt with at the
appropriate time.

To describe tropicalization categorically, we utilize the category
$\cFWNV$ describing fields with  valuation, or, slightly more
generally, the category $\Valdom$ describing integral domains with
valuations. (In the sequel \cite{IKR6} to this paper, we proceed
still further, with valued rings.) The theory is applicable to
fields with valuation, in particular to the Puiseux series field.
Intuitively, the corresponding tropical category just reformulates
the valuation, where the operations are taken from the target of
the valuation.  Our category  $\Valmon$  (cf.~\S\ref{MonVal}) is described in
the language of monoids, in order to permit other tropicalization
techniques.

At this point, let us stress that one principal role of the
 tropical algebra is to provide an intrinsic algebraic setting for
studying valuations in $\cFWNV$ and their extensions, as described
in \S\ref{RedTro}, via  Maslov dequantization \cite{Litvinov2005}
or the degeneration of ``amoebas'' \cite{Einsiedler8311},
\cite{Pass}, and \cite{Shustin1278}. In a few words, one takes the
power valuation to pass from the Puiseux series field to $\Q$,
viewed as the max-plus algebra. This is formalized in
Remark~\ref{6.20} as the functor $\tF_{\operatorname{val}}$ from
the category of valued monoids to  the category of ordered monoids
(or, equivalently, bipotent semirings).

 Unfortunately, the  algebraic theory of bipotent
\semirings0 is too weak to provide much information without
additional structure. Accordingly, the algebra $\maxplus$ was
extended  to \textbf{extended tropical
arithmetic}~\cite{zur05TropicalAlgebra} which evolved into the
\textbf{supertropical domain}
\cite{IzhakianRowen2007SuperTropical} and then to the
\textbf{layered \domain0} $\scrR(L,\tG)$ of an ordered monoid
$\tG$
 with respect to an indexing \semiring0 $L$, called the the
\textbf{sorting set}, cf.~ \cite[Definition~3.5]{IzhakianKnebuschRowen2009Refined}.
$L$-layered \domains0 become max-plus algebras when $L$ is $ \{ 1
 \}$ and become supertropical domains when  $L$ is $ \{ 1,
\infty\}$.

The general $L$-layered theory, set forth in \S\ref{Sec:SupCat}
and \S\ref{supfun}, has many advantages over the other theories,
as shown in~\cite{IzhakianKnebuschRowen2009Refined}, because it
enables us to distinguish among different ghost levels. This is
really a linguistic distinction, as is explained in the next
paragraph. Nevertheless, there is a definite advantage in making
use of the tools available in the language of layered semirings.

Whereas the supertropical domain enables us to distinguish multiple tropical 
roots (say in the polynomial $f(\la) = (\la +3)^2$) from single roots, it does not
 say anything about the multiplicity of the corner root~3. Thus, it would not enable us to
tell intrinsically whether 3 is a root or a pole of the function~$\frac { (\la +3)^j}{ (\la +3)^k}$,
whereas questions of this sort are answered at once in the layered structure. More sophisticated
geometric examples are given in Example~\ref{lay02}.

For the reader's convenience let us point also to several applications
of the layered structure from
\cite{IzhakianKnebuschRowen2009Refined}:

\begin{itemize}
\item \cite[Theorem~8.25]{IzhakianKnebuschRowen2009Refined} The
$\nu$-multiplicativity of the resultant of tropical polynomials
(in one indeterminate). \pSkip

\item \cite[Theorem~8.33]{IzhakianKnebuschRowen2009Refined} The
multiplicativity of the resultant of products of primary tropical
polynomials (in one indeterminate). \pSkip

\item \cite[Theorem~9.8]{IzhakianKnebuschRowen2009Refined} The
computation of the layered discriminant of a tropical polynomial.
\pSkip

\item \cite[Example~10.6]{IzhakianKnebuschRowen2009Refined}
Multiplicity of roots of tropical polynomials, by means of
\cite[Equation~(9.1)]{IzhakianKnebuschRowen2009Refined}. \pSkip

\item \cite[Example~10.8]{IzhakianKnebuschRowen2009Refined} Unique
factorization in many cases, as well as integration being defined.
\end{itemize}

 Thus, we rely heavily on
Construction~\ref{defn50} in order to pass back and forth between
cancellative ordered monoids and $L$-layered \domains0, and this
should be considered the main thrust of the layering procedure.
Note that Construction~\ref{defn50} is formulated for \semirings0
without $\zero$, in order to avoid complications. The more general
theory is given in  \cite{IKR6}.

Intuitively, to obtain the appropriate layered category one might
expect to take morphisms to be \semiring0 homomorphisms that
preserve the layers, and these indeed play a key role to be
described below. The category of main interest for the tropical
algebraic theory is the category $\ULayDom$ of uniform $L$-layered
\bidomains0, which by Theorem ~\ref{fun1}   is isomorphic to the
category $\SOMon$ of cancellative ordered monoids, under the
natural functor that restricts a  uniform $L$-layered \bidomain0
to its submonoid of tangible elements. In this way, we begin to
see how identifications of categories help guide us in developing
the theory.

This leads us to a delicate side issue. Although the ordered
monoids of interest in the tropical theory are cancellative, such
as the real max-plus algebra or, more generally, any ordered
group, homomorphic images of cancellative monoids need not be
cancellative. Thus, for a rich algebraic theory, we need a way of
passing from arbitrary ordered monoids to layered \semirings0.
Unfortunately the naive generalization of Construction
~\ref{defn50} is not a \semiring0 since distributivity fails!  In
order not to go too far afield in this paper,
 we stick with cancellative monoids, and consider noncancellative monoids in \cite{IKR6}.

In \S\ref{RedTro}, we get to the functor $\TropfunoneL:\SValmon
\to \ULayDom $, which describes the passage to the layered
tropical world. $\TropfunoneL$ is applied to fields with
valuation, in particular the field of Puiseux series, and enables
one to translate equality to the ``surpassing relation'' described
in \cite[\S 3.2]{IzhakianKnebuschRowen2009Refined}. In full
generality the functor
 $\TropfunoneL$ involves   subtleties discussed in \cite{IKR6}.

The functor $\TropfunoneL$ is not faithful, since
 it only measures the action of the given valuation, and does not enable
 us to distinguish among elements having the same value. Thus in \S\ref{FulTro} we also
 introduce briefly the \textbf{exploded tropicalization functor}
 utilized in \cite{Par} and another functor $\Tropfuntwo$   which retains extra information given in Proposition~\ref{fullfun},
 such as  is contained in ``coamoebas.'' Borrowing from classical
 valuation theory, we describe the exploded tropicalization
 functor in terms of the associated graded algebra, noting that in
the case of the field of Puiseux series, the components of the
associated graded algebra can be identified with the base field.

 Although monoids
recently have been seen to provide much of the underpinning for
algebraic geometry, cf.~\cite{Berk2011,CHWW,Pat1,Pat2} for
example, classical algebraic geometry relies for a large part on
roots of polynomials, which can be understood more easily using
\semirings0. Our approach to tropical geometry is to define affine
varieties as sets of ``ghost roots'' of
  polynomials.

 As is well known, and discussed in detail
in~\cite{IzhakianRowen2007SuperTropical}, in contrast to the
classical situation for polynomials over algebras over an infinite
field, different tropical polynomials over a \semiring0 $R$ often
take on the same values identically, viewed as functions.
Furthermore, in max-plus situations one often wants to use
variants such as Laurent polynomials (involving $\la^{-1}$ as well
as $\la$) or polynomials with rational exponents, or even more
generally one could talk in the language of the lattice of
characters and its dual lattice, cf.~\cite[\S 2.2]{Payne08}. Also,
as in classical algebraic geometry, we often want to limit the
domain of definition to a given subset of $R^{(n)}$ such as
an algebraic variety. Thus, we work directly with functions from a
set~$\tSS$ to a layered domain, denoted $\Fun (\tSS,R)$, or, more
specifically, polynomially defined functions, denoted $\Pol
(\tSS,R)$ or Laurent polynomially defined functions, denoted $\Lau
(\tSS,R)$. In Proposition~\ref{funfunc} we check that passing to
the monoid of functions from $\tSS$ to an ordered monoid $\tM$ and
then translating to \semirings0 yields the same categorical theory
as moving first to a bipotent \semiring0 and then passing to its
function \semiring0.

Thus, taking $\tSS\subset R^{(n)}$, we redefine polynomials and
monomials over $R$ intrinsically as functions from $\tSS$ to $R$,
 leading to an analog of the Zariski topology in Definition \ref{Zar1}.
  This enables us to define a coordinate \semiring0 via  Definition \ref{coord2}.

Our view of tropical geometry relies largely on Kapranov's
  Theorem, as extended by~\cite{Payne09}, which describes
the 1:1 correspondence between roots of polynomials and
  corner roots of their tropicalizations. This process is understood
  categorically in terms of the supertropical structure, in
  \S\ref{Kapr}. Ironically, although the
  Kapranov-Payne Theorem can be stated in the language of the
tropicalization functor  $\TropfunoneL$, the exploded
tropicalization functor is   needed (at least implicitly) in order
to carry out the proof.

\subsection{Overview of the major categories and functors in this paper}$ $

In summary, let us review the main algebraic categories and their
uses.

\begin{itemize}
\item The category $\cFWNV$ (resp.~$\Valmon)$ describing fields
with valuation (resp.~integral domains) with  valuation. This is
the ultimate arena of investigation, but its theory often is very
difficult, thereby historically giving rise to tropical
mathematics, which can be thought of as a degeneration of
$\cFWNV$. \pSkip

\item  The category $\SOMon $ of  cancellative ordered monoids,
which is isomorphic to the category $\Bipot$ of bipotent semirings. This is
the traditional algebraic category underpinning tropical
mathematics, but is too coarse a degeneration for many algebraic
arguments, and often requires returning to $\cFWNV$ in proofs.
\pSkip

\item   The category $\ULayDom$ of uniform $L$-layered \bidomains0. For $L = \{ 1 \}$
this is just the category $\SOMon $. For  $L = \{ 1 , \infty \}$
or $L = \{ 0, 1 , \infty \}$ we get the supertropical theory,
which suffices in linear algebra for the investigation of
nonsingular matrices, bases, characteristic polynomials, and
related notions. In order to discuss multiple roots of polynomials
and singularity of curves, one needs to take $L$ containing
$\Net.$ \pSkip

\item  The category of exploded $L$-layered \bidomains0. This is
used by Sheiner and Shnider for the proof of the Kapranov-Payne
theorem, as well as other deep results in the theory.
\end{itemize}

Here is a diagram of the categories under discussion in this
paper, and the main functors connecting them.


$$
\xymatrix @R=2.5pc{
 \SValmon   \ar @{>}[dr]^{\TropfunoneL}   \ar @{>}[r]^{\tF_{\operatorname{val}}} 
 &     \SOMon
    \ar @{>}[dr]^{\Fun_{\Mon}(\tSS,\unscr \,  )} \ar @{>}[d]^{\tFlay} \ar @{<->}[r]^{\tF_{\OMon}}  & \Bipot \ar@{}[r]|-*[@]{\subset}     & \SeRo   \ar@{>}[d]^{\Fun_{\SeRo}(\tSS,\unscr \,  )}
  \\ \Valdom \ar@{}[u]|-*[@]{\bigcup}  \ar @{>}[rd]^{\Tropfunexp}   &
\ULayDom    &  \SOMon \ar @{>}[r] & \SeRo
 \\
 \Valfld \ar@{}[u]|-*[@]{\bigcup}   &
  \ULayDom \times \Ring  \ar@{>}[u]  &
& \\
  }
$$

$\tF_{\operatorname{val}}$, which can be viewed as the customary tropicalization procedure,   formalizes the order valuation on Puiseux series, which in most recent research has replaced the logarithm as the means of tropicalizing a variety.

$\TropfunoneL$, perhaps the most important functor in our theory,  takes us from the classical world of algebraic geometry
to the layered tropical world of this paper.

$\tF_{\OMon}$ is the functor that enables us to pass from ordered monoids to bipotent semirings, thereby putting tools
of semiring theory (such as polynomials) at our disposal.

$\tFlay$  is the functor that enables us to ``layer'' an ordered monoid, and thus pass to the layered theory.

$\Tropfunexp$ is the ``exploded'' functor, which preserves the leading coefficient of the original polynomial when tropicalizing, and thus permits Payne's generalization of Kapranov's theorem (and its application to tropical varieties).

The Fun functors take us to  \semirings0 of functions, thereby enabling us to treat polynomials (as functions).

 At the conclusion of
this paper, we consider how the layered category $\LayDom$ enables
us to define \textbf{corner varieties}, and we relate the
algebraic and geometric categories along the classical lines of
algebraic geometry, obtaining a Zariski-type correspondence in
Proposition~\ref{Zarcor1}.

There also is a functor $\Tropfuntwo:  \Valdom \to \ULayDom
\times {\Valmonone}$ which we did not put into the diagram, whose justification is given in the discussion after Proposition~\ref{fullfun}.

%


\section{Background}

We start with the category $\Mon$ of monoids and their monoid
homomorphisms, viewed in the context of universal algebras, cf.~
Jacobson~\cite[\S 2]{Jac}.

\begin{defn} A \textbf{semigroup} is a set with an associative operation,
usually written multiplicatively as~$\cdot$. A \textbf{monoid}
$\tM:= (\tM,\cdot \ )$ is a semigroup with a unit element
$\one_\tM$. A semigroup $\tM$ is \textbf{(left) cancellative} with
respect to a subset $S$ if for every $a_1,a_2,   \in \tM$, $ b \in
S$,
$$ b \cdot a_1 = b \cdot a _2 \quad \text{ implies} \quad  a_1 =
a_2.$$ $\tM$ is \textbf{cancellative} if $\tM$ is   cancellative with respect to itself.

An element $a$ of $\tM$ is \textbf{absorbing} if $ab = ba = a$ for
all $b\in \tM$. Usually the absorbing element (if it exists) is
denoted as the \textbf{zero element} $\zero_\tM$, but it could
also be identified with $-\infty.$ A semigroup $\tM$ is
\textbf{pointed} if it has an absorbing element $\zero_\tM$. A
pointed semigroup $\tM$ is \textbf{cancellative} if $\tM$ is
cancellative with respect to $\tM \setminus  \{ \zero_\tM \}. $ A subset
$\mfa \subset \tM $ is a \textbf{left (right) semigroup ideal} if
$\tM \mfa \subset \mfa$ ($\mfa \tM  \subset \mfa$). The semigroups (as well as monoids) in this paper
are presumed commutative, so left semigroup ideals are semigroup ideals.

A semigroup $\tM := (\tM, \cdot \,)$ is \textbf{divisible}
if for every $a \in \tM$ and $m\in \Net$ there is $b\in \tM$ such
that $b^m =a.$ 

 A \textbf{semigroup homomorphism} is
a map $\phi : \tM \to \tM'$ satisfying $$ \phi  (a_1 a_2) = \phi
(a_1)\phi (a_2), \quad \forall a_1,a_2 \in \tM.$$ (When dealing
with pointed semigroups, we also require that $\phi (\zero_\tM) =
\zero'_\tM$.)
  A \textbf{monoid homomorphism} is
a semigroup homomorphism $\phi : \tM \to \tM'$ also satisfying $
\phi (\one_\tM) = \one_{\tM'}  .$
\end{defn}

\subsection{Semirings without zero}\label{sem0}

We ultimately work in the environment of semirings (or, more
precisely, semirings without a zero element, which we call a
\textbf{\semiring0}. A standard general reference for the
structure of semirings is \cite{golan92}; also cf.~\cite{Cos}.
Thus, a \semiring0 $(R,+,\cdot \, , \rone)$ is a set $R$ equipped
with two binary operations $+$ and~$\cdot \;$, called addition and
multiplication, together with a \textbf{unit element}~$\rone$ such
that:
\begin{enumerate}
    \item $(\SR, +)$ is an Abelian semigroup; \pSkip
    \item $(\SR, \cdot \ , \rone )$ is a monoid with unit element
    $\rone$; \pSkip
    \item Multiplication distributes over addition on both sides. \pSkip
\end{enumerate}
When the multiplicative monoid $(R, \cdot \, , \rone )$ is
cancellative, we say that $\SRo$ is a \textbf{\domain0}; when $(R,
\cdot \, , \rone )$ is also an Abelian group, we say that $\SRo$
is a \textbf{\semifield0}. As customary, $\Net$ denotes the
positive natural numbers, which is a cancellative \domain0.

\begin{defn}
A \textbf{homomorphism} $\srHom : R \to R'$  between two
\semirings0 is defined as a homomorphism of multiplicative monoids
that also preserves addition, to wit,
$$\text{ $\srHom(a + b) = \srHom(a) + \srHom(b)$  \qquad for all
$a,b \in R$.} $$ A \semiring0 \textbf{isomorphism} is a \semiring0
homomorphism that is 1:1 and onto.  \end{defn}

Thus, we have the category $\SeRo$ of \semirings0 and their
homomorphisms. This is closely related to the category $\SeR$ of
semirings and  semiring homomorphisms, especially since the
semirings $R$ of interest in the tropical theory, besides being
multiplicatively cancellative, have the property that $a+b \ne
\rzero$ unless $a = b = \rzero$; in other words, $R\setminus \{ \rzero \}$ is closed under
 addition.

\begin{rem}\label{embed0} Any \semiring0 $R$ can be embedded in a
semiring $R\cup \{\zero\}$ by formally adjoining a zero element
~$\zero$ satisfying $\zero + a = a + \zero = a$ and $\zero  \cdot
a = a \cdot \zero = \zero ,$ $\forall a \in R\cup \{\zero\}$.
Conversely, if $R$ is a semiring such that $R\setminus \{ \rzero \}$ is closed
under multiplication and addition, then $R\setminus \{ \rzero \}$ is a
\semiring0.\end{rem}

\begin{prop} The category $\SeRo$ is isomorphic to a subcategory of the
category $\SeR$.\end{prop}
\begin{proof} We just apply Remark~\ref{embed0}, noting that any
 \semiring0 homomorphism $\vrp: R \to R'$ can be extended to
 a semiring  homomorphism $\vrp: R\cup \{\zero_R\} \to R'\cup \{\zero_{R'}\}$
 by putting $\vrp(\zero_R) = \zero_{R'}.$
\end{proof}

 An \textbf{ideal} $\mfa$ of a \semiring0  $R$ is
defined  to be a sub-semigroup
 of $(R,+)$ which is also a (multiplicative) semigroup ideal of~$(R,\cdot \, , \rone)$.
(Clearly, when $R$ has a zero element $\rzero$, then $\rzero \in
\mfa$.)
%
%
%

\begin{example}  If $R$ is a semiring, then $\{ \rzero\}$ is an ideal of the semiring $R\cup \{ \rzero \}$ of
Remark~\ref{embed0}.
\end{example}

%

The tropical theory is closely involved with certain kinds of \semirings0.
 \begin{defn}\label{bip}   A  \semiring0 $R$ is \textbf{idempotent} if $$a+a = a  \qquad \forall a\in R;$$  $R$ is \textbf{bipotent} if $$a+b \in \{
a, b\} \qquad \forall a,b\in R.$$
\end{defn}

The max-plus algebra is the prototype of a bipotent \semiring0.

\subsection{Congruences}
Unfortunately, kernels, such an important feature of category
theory, play virtually no role in the general structure theory of
semirings. In ring theory, the kernel $\vrp^{-1}(\zero_{R'})$ of
any onto homomorphism $\vrp: R \to R'$ is an ideal $\mfa$ of $R$,
and furthermore one can recover $R'$ as isomorphic to $R/\mfa$.
This is not the case with \semirings0. Ideals do not play such a
powerful  role in the structure theory of \semirings0, since the
construction $R/\mfa$ is problematic for an arbitrary ideal $\mfa$
(the difficulty arising from the fact that distinct cosets need
not be disjoint).

Instead, one needs to consider more generally
equivalence relations  preserving the \semiring0 operations.
From the general theory of universal algebra, one defines a
 \textbf{congruence} $\Cong $ of an algebraic structure $\mathcal A$ to be an equivalence relation
 $\equiv$ which preserves
 all the relevant operations and relations;  we call  $\equiv$ the \textbf{underlying  equivalence} of $\Cong $. Equivalently, a congruence $\Cong$ is a sub-\semiring0 of $\mathcal A\times \mathcal A$ that contains the diagonal $$\diag (\mathcal A):= \{ (a,a): a \in \mathcal A \}$$ as described in Jacobson~\cite[\S
 2]{Jac}.  In other words,
 writing the underlying equivalence relation as $a \equiv b$ whenever $(a,b) \in \Cong,$
we require that $\equiv$ preserves
 all the relevant operations and relations.

 \begin{rem}\label{cong10} We recall some key results of \cite[\S
 2]{Jac}:

 \begin{itemize}\item Given a congruence  $\Cong$ of an algebraic
 structure $\mathcal A$, one can endow the set $$\deqvb{\tA} := \{ [a]: a \in \mathcal A \}$$ of equivalence
 classes  with the same (well-defined) algebraic structure, and the map
 $a \mapsto [a]$ defines an onto homomorphism $\mathcal A\to\deqvb{\tA}$. (For
 this reason,
 Berkovich~\cite{Berk2011} calls them ``ideals,'' but  this
 terminology conflicts with some of the literature, and we prefer
 to reserve the usage of ``ideal'' for the usual connotation.)
 \pSkip

\item  For any homomorphism $\vrp:\mathcal A \to \mathcal A',$ one
can define a congruence $\Cong $ on $\mathcal A$ by saying that $a
\equiv b$ iff $\vrp(a) = \vrp (b).$ Then $\vrp$ induces a 1:1
homomorphism
$$\tlvrp :\deqv{\tA}  \ds \to  \mathcal A',$$ via $\tlvrp ([a]) =
\vrp(a)$.
\end{itemize}
 \end{rem}

We repeat the definition of congruence in each specific case  that
we
 need. Thus, a  congruence $\Cong$ on a semigroup $\tM$ is an equivalence relation
that preserves multiplication, in the sense that if $a_1 \equiv
b_1$ and $a_2 \equiv b_2,$ then $a_1a_2 \equiv b_1 b_2.$
 In this case, the set of
equivalence classes $\tM/\Cong $ becomes a semigroup under the
operation
$$[a][b] = [ab],$$ and there is a natural semigroup
 homomorphism given by $a \mapsto [a].$ When $\tM$ is a monoid,
 this becomes a monoid homomorphism, since $[\one_M]$ is the
 multiplicative unit of  $\deqvb{\tM}$. When $\tM$ is
pointed, then $\deqvb{\tM}$ is also pointed, with absorbing
element $[\zero_\tM].$

Here is another instance of a congruence that comes up in the
passage from arbitrary monoids to cancellative monoids.

\begin{example}\label{Pass011}
Given an equivalence relation $ \equiv$  on a semigroup $\tM$, and
a sub-semigroup $S$ of $\tM$, we define the equivalence  $
\equiv_S$ given by $b_1 \equiv_S b_2$ if $b_1s \equiv  b_2s$
 for some $s\in S.$ When $ \equiv$  defines a congruence~ $\Cong$, then $ \equiv_S$
 also defines a congruence $\Cong_S.$   This congruence then identifies
$b_1$ and $b_2,$ thereby eliminating instances of
non-cancellation, and is a useful tool.
\end{example}


\subsubsection{Congruences over \semirings0}

The congruence $\Cong$ is a \semiring0 congruence on a \semiring0
$R$ iff
\begin{equation}\label{cong1}
a_1 \equiv a_2 \text{ and } b_1 \equiv b_2 \quad  \text{imply}  \quad  \left
\{
\begin{array}{lll}
a_1 +b_1 & \equiv &  a_2+b_2, \\a_1  b_1 &  \equiv & a_2 b_2.
\end{array}
\right.\end{equation}

 \begin{lem}\label{Pass2} To verify the conditions in  \eqref{cong1} for commutative \semirings0, it is enough to
assume $b_1 = b_2$ and show for all $a_1,$ $a_2,$ and $b$ in
$R$:

\begin{equation}\label{cong11}
a_1 \equiv a_2   \quad  \text{implies}  \quad
a_1 +b \equiv    a_2+b; \end{equation}

\begin{equation}\label{cong12}
a_1 \equiv a_2   \quad  \text{implies}  \quad
a_1  b \equiv   a_2 b. \end{equation}
 \end{lem}
\begin{proof} $a_1 +b_1 \equiv    a_2+b_1  \equiv    a_2+b_2.$
Likewise,
  $a_1 b_1 \equiv    a_2b_1  \equiv    a_2b_2.$
\end{proof}
It often turns out that  \eqref{cong11}  enables us to obtain
\eqref{cong12}. On the other hand, in the case of \semifields0,
multiplicative cosets are more easily described  than additive
cosets, as is described in detail in~ \cite{HW}. To wit, let $N :=
\{ a\in R: a \equiv \rone \}.$ For any $a \in R$
we have $ab_1 \equiv ab_2$ iff $b_1 b_2^{-1} \in N.$

%

We write $\diag(\tM)$ for $\{ (a,a): a \in \tM \}$. As
Berkovich~\cite{Berk2011} points out, any semigroup ideal $\mfa$ of a
semigroup $\tM$ gives rise to the congruence $( \mfa \times \mfa ) \cup
\diag (\tM),$ which corresponds to the Rees factor semigroup, and
the analogous statement holds for monoids. A wrinkle emerges when
we  move to bipotent \semirings0, since $( \mfa \times \mfa) \cup
\diag(\tM)$ need not be closed under addition. Thus, the
applications are limited, and are discussed in \cite{IKR6}.

\begin{defn} An \textbf{identity} $f = g $ of a \semiring0 $R$ is an elementary sentence $f(x_1,\dots, x_m)  = g(x_1,\dots, x_m) $ that holds for all $x_1,\dots, x_m$  in $R$.
\end{defn}
\begin{rem}\label{Pass00}
Suppose we want to force a  \semiring0 $R$ to satisfy a particular
identity, in the sense that we want a  \semiring0 $ \bar  R$ in
which $f = g $ is an identity, together with a surjective
homomorphism $\varphi: R \to  \bar R$ satisfying the universal
property that any homomorphism of $R$ to a \semiring0 satisfying
the identity $f = g $ factors through $\varphi$.

Intuitively, one must mod out the relation  $f = g $ by putting
$f(a_1,\dots, a_m) \equiv g(a_1,\dots, a_m) $ for all $a_i$ in $R$.
 For  \semirings0, in view of Lemma~\ref{Pass2}, since we are dealing with congruences,
we must mod out the equivalence relation obtained  by putting
$f(a_1,\dots, a_m) +c  \equiv g(a_1,\dots, a_m) +c $ and
$f(a_1,\dots, a_m) c  \equiv g(a_1,\dots, a_m) c $  for all $c$
and $a_i$ in $R$.
\end{rem}

\begin{example}\label{Pass000}
Consider the  \textbf{additive idempotence} identity  $x+x = x.$
We attain this by imposing the equivalence relation   given by
$a+a \equiv a, \forall a \in R.$ The congruence that it generates
must also satisfy the relation $a+a +c \equiv a+c, \forall a \in
R.$

But then we also get  \eqref{cong12}, since  $ab + ab = (a+a)b$
and
$$(a+a+c)b = ab + ab + cb \equiv ab+cb = (a+c)b
 .$$ Thus, \eqref{cong11} already defines the congruence. (This observation is to be elaborated
shortly.)

Note that additive idempotence implies all identities of the form $x+x + \dots + x = x.$
\end{example}

 Whereas in ring theory the equivalence class $[\rzero]$ determines
 the
congruence~ $\Cong$, this is no longer the case for semirings, and
we need to consider all the classes $\{ [a]: a \in R\}.$ This is
another reason that we do not require the element $\zero$ in a
\semiring0, for it has lost much of its significance.
Nevertheless, ideals do play a role in the layered algebraic
theory, pursued in a different
paper~\cite{IzhakianRowen2011Ideals}.

\begin{lem}\label{catid1} The bipotent \semirings0 comprise a full
subcategory $\Bipot$ of $\SeRo$.\end{lem}\begin{proof} If
$\varphi: R \to R'$ is a \semiring0 homomorphism and $a+b \in
\{a,b\},$ then
$$\varphi (a) +\varphi (b)= \varphi (a+b) \in \{\varphi (a),\varphi
(b)\}.$$
\end{proof}

\subsection{Hom and Adjoint functors}

We need to use some well-known facts about categories.

\begin{defn}\label{hom0}
For any category  $\mfC$ and  some given object $A$ in $\mfC$, we
recall the well-known covariant functor
$$\hom(A,\underline{\phantom{M}}\,): \mfC \to \SET,$$  which sends
an object $B $ in $\mfC$  to $\hom (A,B),$  and which sends the
morphism $\phi: B \to B'$ to the morphism $\Hom(\unscr ,\phi ):
\hom(A ,B)\to \hom(A ,B')$ given by $f\mapsto \phi f$ for $f: A
\to B.$

Likewise, given an object $B$ in $\mfD$,  we define the
contravariant functor
$$\hom(\underline{\phantom{M}}\, ,B): \mfD \to \SET$$ which sends an object $A$ to $\hom (A,B),$
and which sends the morphism $\vrp: A' \to A$ to the morphism
$\Hom(\vrp,\unscr \ ): \hom(A' ,B)\to \hom(A,B )$ given by $ f
\mapsto f \vrp$ for $f: A \to B.$
\end{defn}

 Recall that a functor $\tF:
\mfC \to \mfD$ is a \textbf{left adjoint} to $\tH: \mfD \to \mfC$
 (and $\tH$ is a \textbf{right adjoint} to $\tF$)
 if there is a canonical identification $\Psi:\Hom( \tF (A),B) \to \Hom (A,\tH(B))$
 for all objects $A$ of $ \mfC $ and $B$ of  $ \mfD $, for which the
 following diagrams are always commutative for all morphisms $\vrp : A \to A'$ and $\phi: B \to B'$:
 $$
\xymatrix{ \Hom (\tF(A),B)  \ar @{>}[d]^{\Psi }  \ar
@{>}[rr]^{\Hom(\tF(\vrp),\unscr \ )} & & \Hom (\tF(A'),B)
\ar @{>}[d]^{\Psi }   \\
 \Hom (A,\tH(B)) \ar @{>}[rr]^{\Hom(\vrp,\unscr \ )}  &  &  \Hom (A',\tH(B))  ,}
$$
 $$
\xymatrix{ \Hom (\tF(A),B)  \ar @{>}[d]^{\Psi }  \ar
@{>}[rr]^{\Hom(\unscr ,\phi)} & & \Hom (\tF(A),B')
\ar @{>}[d]^{\Psi }   \\
 \Hom (A,\tH(B)) \ar @{>}[rr]^{\Hom(\unscr , \tH(\phi))}  &  &  \Hom (A,\tH(B'))  .}
$$

It is well-known that any left adjoint functor is unique up to
isomorphism.

\subsection{Universals}

 Recall from \cite[\S1.7 and~ \S1.8]{Jac} that the adjoint functor of a functor $\tF:
\mfC \to \mfD$ is obtained by identifying the appropriate
universal $U$ of  $\tF$ together with the canonical morphisms
$\iota: D \to \tF(U(D))$, for objects $D$ in $\mfD,$  satisfying
the property that for any morphism $f: D\to \tF(C)$ in~$\mfD $ and
object $C$~ in~$\mfC$, there is a morphism $g: U(D) \to C$ in
$\mfC $ such that
$$\tF(g \circ \iota) = f.$$ The example used in \cite{CHWW} is
for the forgetful functor from $K$-algebras to monoids; its
universal is the monoid algebra $K[\tM]$ of a monoid $\tM$.

\begin{example}\label{Pass0}  We  define the forgetful functor $ \SeRo \to \Mon,$ by
forgetting addition. The appropriate universal in this case is the
\textbf{monoid \semiring0} $\Net [\tM],$ defined analogously to
the  monoid algebra. \pSkip
\end{example}

\begin{defn}\label{monoidsem} $[\tM]$ denotes the \semiring0
obtained by taking $\Net [\tM]$ modulo the additive idempotence congruence
of Example~\ref{Pass000}. Explicitly, $[\tM]$ is comprised of
formal sums of distinct elements of the monoid  $\tM$, i.e.,
$$[\tM] = \bigg\{ \sum _{a\in S} a \ds {|}  S \subset \tM \bigg
\}$$ endowed with addition $$\sum _{a\in S} a + \sum _{a\in S'} a
= \sum _{a\in S\cup S'} a \ ,$$ and multiplication is obtained
from the original multiplication in $\tM$, extended
distributively. \pSkip
\end{defn}

\begin{example}\label{Pass1} 
Since additive idempotence defines an identity, one has the
category of additively idempotent \semirings0; the forgetful
functor to $\Mon$ now has the universal $[\tM]$.
The customary way to view tropical mathematics is by means of the
max-plus \semiring0, which is additively idempotent.\end{example}

\section{Pre-ordered semigroups, monoids, and semirings}

Recall that a \textbf{partial pre-order} is a transitive relation
$( \le )$; it is called a \textbf{partial order} if it is
antisymmetric, i.e., $a \le b$ and $ b\le a$ imply $a = b$.  We write $a<b$ when $a \le b$ but $a \ne b.$

A partial pre-order is
called a  \textbf{preorder} if any two elements are
comparable.
A \textbf{(total) order} is a partial order which is also
 a preorder.

\subsection{Pre-ordered semigroups}

We work with pre-ordered semigroups in this paper. The natural
 definition in terms of universal algebra is the following:

\begin{defn}\label{ordered1} A semigroup $\tM := (\tM, \cdot \, )$ (or a monoid $\tM := (\tM, \cdot \, , \one_\tM)$) is  \textbf{partially
pre-ordered} (resp.~\textbf{partially ordered},  \textbf{pre-ordered}, \textbf{ordered})
if it has a  partial pre-order $\le$ (resp. ~partial  order,  pre-order,
order) such that
\begin{equation}\label{dist10} b\le c \quad \text{implies}
\quad ab \le ac \quad \text{and}  \quad ba \le ca, \quad \forall a
\in \tM.\end{equation}\end{defn}
We denote an ordered semigroup by $\otM$. Thus, totally ordered
semigroups satisfy the following property:
\begin{equation}\label{dist1} a \max \{ b, c \} = \max \{ ab,
ac\}, \qquad \forall a,b,c \in \tM.\end{equation} We say that the
relation $(\leq)$ is \textbf{strict} if
\begin{equation}\label{dist12} b< c \quad \text{implies} \quad ab <ac \quad \text{and}
 \quad ba <
ca, \quad \forall a \in \tM.\end{equation}

 \begin{Note}
This definition
 requires that all elements of $\otM$ are positive or 0, an implicit assumption made throughout this
 paper, to  be discussed after Definition~\ref{ordered2}.\end{Note}

\begin{lem} A total  order $(\le)$ on a  semigroup $\tM$ is strict iff the  semigroup $\tM$ is
cancellative. \end{lem} \begin{proof} $(\Rightarrow):$ Suppose $ab
= ac$. By symmetry, we may assume that $b \le c.$ But if $b < c$
then $ab < ac$, a contradiction, so  we conclude $b =c $.

$(\Leftarrow):$ If $b < c$ then $ab \ne ac$, implying $ab < ac$.
 \end{proof}

Let us construct the appropriate categories.

\begin{defn}
 An \textbf{order-preserving} semigroup homomorphism is a
 homomorphism $\phi:  \tM  \to \tM'$ satisfying the
condition (where $\le$ denotes the partial order on the
appropriate semigroup):
\begin{equation}\label{comm12} a \le b \quad \text{implies}  \quad \phi (a)\le \phi (b),
\quad \forall a,b \in \tM.\end{equation}

 $\PrePOMon $, $\PreOMon $,  $\POMon $, $\SPOMon $,  $\OMon $, and $\SOMon $  denote the respective categories
 of partially
pre-ordered, pre-ordered, partially ordered, cancellative
partially ordered, ordered, and cancellative ordered monoids,
whose morphisms are the order-preserving homomorphisms.
\end{defn}
By definition, $\SOMon $ is a full subcategory both of $\OMon$ and
of $\SPOMon $, each of which is a full subcategory of $\POMon $,
which is a full subcategory of $\PrePOMon $.

%

\begin{rem}\label{forg1}
The forgetful functor from  the category $\POMon $ to  the
category $\PrePOMon$ is obtained by viewing any partially ordered
monoid $\tM$ naturally in $\PrePOMon$.
\end{rem}

We also can go in the other direction.

 \begin{rem}\label{indord}
 For the class of partially pre-ordered semigroups, our congruences $\Cong$  also
 satisfy the property that if $a_1 \le a_2$ and $b_i \equiv a_i,$
 then $b_1 \le b_2$. In this case, $\deqvb{\tM}$ inherits the
 partial pre-order given by $[a] \le [b]$ iff $a \le b.$\end{rem}

 \begin{prop}\label{equiv1}
There is a retraction  $\tF:\PrePOMon\to \POMon
 $ to the forgetful functor of Remark~\ref{forg1}. Namely, we take the congruence  $\Cong$ on a  pre-ordered monoid $\tM$ given  by
 $a \equiv b$ when  $a \le b$ and $b \le a$, and
define $\tF(\tM) := \deqvb{\tM}$. \end{prop}
\begin{proof}  It
is easy to see that $\equiv$ is an equivalence relation that
preserves the  operation and the order, so is an ordered monoid
congruence, and thus induces a partial order on $\deqvb{\tM}$
according to Remark~\ref{indord}.

We claim that any order-preserving homomorphism $\phi: \tM \to
\tM'$ induces an order-preserving homomorphism $\tlphi:
\deqvb{\tM} \ds \to  \deqvb{\tM'}.$ Indeed, if $a \equiv b$, then
$a \le b$ and $b \le a$, implying $\phi(a) \le \phi(b)$ and
$\phi(b)
 \le \phi(a)$, yielding  $\phi(a) \equiv \phi(b)$.

The functor $\tF$ is a retraction to the forgetful functor, since
it acts trivially on any total ordered monoid.
 \end{proof}

\subsection{Pre-ordered \semirings0}

\begin{defn}\label{ordered2} We say that a \semiring0   $R$ is  \textbf{pre-ordered} (resp.~\textbf{partially ordered}, \textbf{ordered})
if it has a  pre-order $\ge$ (resp. ~partial  order,
order) with respect to which  both the monoid $(R,\cdot \,
,\rone)$ and the semigroup $(R,+)$ satisfy Condition
\eqref{dist10} of Definition~\ref{ordered1}. \end{defn}

In other words, both multiplication and addition preserve the
pre-order. There is a delicate issue in this definition. In the
rational numbers, viewed as a multiplicative monoid, we have $1
<2$ but $(-1)1 > (-1)2.$  This difficulty is dealt with in
\cite{IzhakianKnebuschRowen2009Refined}, in which we define the
order in terms of a cone of ``positive''  elements.
Definition~\ref{ordered2} is the special case in which  all
elements of $R$  are positive or 0, and is  reasonable for
tropical mathematics since the ``zero'' element (when it is
included) is minimal.  We use Definition~\ref{ordered2} here
because it is more appropriate to our categorical treatment.

 A \semiring0 $R$   has the \textbf{infinite element} $\infty$
if
\begin{equation}\label{inf} \infty + a = \infty = \infty\cdot a =
a\cdot \infty, \qquad \text{for some } a \in R.
\end{equation}
Recall from \cite[Corollary
2.15]{IzhakianKnebuschRowen2009Refined} that if $R$ has a unique
infinite element, then   $$\infty + a = \infty = \infty\cdot a =
a\cdot \infty, \qquad \forall a <\infty.$$

Nonzero positive elements of an ordered \semiring0 need not be
finite, and we could have several infinite elements (as can be
seen easily by means of ordinals). We do not deal with such issues
in this paper, and assume there is at most one infinite element
$\infty$.

The following observation is implicit in~\cite[Theorem~4.2]{HW}.

 \begin{prop}\label{maxplus} There is a natural functor   $\SeRo \to
\PrePOMon,$ where we define the preorder on a \semiring0 $R$ given
by $$\text{$a \le b$ \dss{iff}  $a=b\quad$ or $\quad b = a+c \quad$
 for some $c \in R.$}$$
\end{prop}

 This functor always yields the trivial partial
preorder on rings, since then $b =a + (b-a).$ The situation is
quite different for the \semirings0 arising in tropical
mathematics, because of bipotence.

\begin{prop}\label{maxplus20} Suppose $R$ is an idempotent
\semiring0.
\begin{enumerate}\eroman
\item $a\le b$ iff $a + b = b.$

\item   $\le$ is a  partial order, which is total when $R$ is bipotent.
\end{enumerate}\end{prop}\begin{proof} (i):
$(\Leftarrow)$ Take $c=b.$

$(\Rightarrow)$  Suppose $a + c = b.$ Then
$$ a+b = a + (a+c) = (a+a)+c = a+c = b.$$

 (ii):
Transitivity follows because $a+b =b$ and $b+c = c$ imply
$$a+c = a+(b +c) = (a+ b)+c = b+c = c.$$

 It remains to prove
 antisymmetry. Suppose $a\le b$ and $b\le a.$ Then, in view of
 (i),
 $b = a+b = a$.
\end{proof}

We are ready for a key identification of categories.

\begin{prop}\label{maxplus1} There is
 a faithful functor $\tF_{\OMon}: \OMon \to \SeRo$, whose image  is $\Bipot$. \end{prop}
\begin{proof}
Given any  totally ordered monoid $(\tM, \cdot \, , \ge, \mone )$
we define $a+b$ to be $\max \{ a, b\}.$ Then $(\tM,+)$ is a
semigroup, since
$$(a+b)+c = \max \{ a, b , c\} = a+(b+c).$$
Finally, this gives rise to a \semiring0, since, by
Equation~\eqref{dist1}, $$(a+b)c = \max \{a,b\}c = \max \{ ac, bc
\} = ab + ac.$$ Any order-preserving monoid homomorphism $\vrp:
\tM \to \tM'$ is a \semiring0 homomorphism, since for $a\le b$ we
have
$$\vrp(a+b) = \vrp(b)=\vrp(a) + \vrp(b).$$

 Conversely, given a bipotent
\semiring0 $R$,  the relation $\le$ of Proposition~\ref{maxplus}
is a total order by Proposition~\ref{maxplus20}. Furthermore,
$\le$ is preserved under multiplication, since $b \le c$ implies
$b+c =c$ and thus $ab + ac = a(b+c) = ac,$ yielding $ab \le ac.$
Any \semiring0 homomorphism $\vrp: R \to R'$ is an
order-preserving monoid homomorphism, for if $a \le b$, then
$$\vrp(a) + \vrp(b) = \vrp(a+b) = \vrp(b),$$ implying
$\vrp(a) \le \vrp(b).$
\end{proof}
 Note that we have just reconstructed the max-plus algebra.
 We will rephrase this result in the layered
setting, as Theorem~\ref{fun1}.

Proposition~\ref{maxplus1} enables one to pass back and forth
between categories of totally ordered monoids and bipotent
semirings.
 The first category enables us to exploit techniques from
 valuation theory, whereas the second enables us to introduce
 concepts from ring theory such as polynomials, modules, matrices, and
 homology theory.

\begin{prop}\label{rmk:duality} There is a functor $\Bipot\to
\Bipot$ sending a \semiring0 to its \textbf{dual} bipotent
\semiring0 obtained as the same multiplicative monoid, but
reversing the bipotence in addition; i.e., if originally $a+b =
a$, now we put $a+b = b.$ \end{prop}
\begin{proof} This is seen readily by defining the pre-order given
by $a \ge  b$ iff $a+b = a;$ then the dual bipotent \semiring0
corresponds to the reverse pre-order, and any homomorphism
preserves the (reverse) order.
\end{proof}

For example, the dual \semiring0\ of $(\Real , \max, + , 0)$ is
$(\Real , \min, + , 0)$. (The number $0$ is really the unit
element~$\one_\Real$.)

 \subsection{The universal for the Frobenius property}

 Usually one works with commutative, totally ordered monoids and
\semirings0. In this case,
 recall  the well-known \textbf{Frobenius property}, cf.~ \cite[Remark~1.1]{IzhakianRowen2007SuperTropical}:
\begin{equation}\label{eq:Frobenius} (a+b)^m = a^m + b^m
\end{equation}
for any
    $m \in \Net$.

 These  are identities of $R$, so we could define
    the \textbf{Frobenius monoid \semiring0} $\httM$ of an arbitrary monoid $\tM$, in which, in view of
    Remark~\ref{Pass00}, we
    impose on $[\tM]$ (defined in Definition~\ref{monoidsem}) the relations
    $$ \bigg( \sum _{a\in S} a\bigg)^m +c  = \bigg(\sum_{a\in S} a^m\bigg) +c , \qquad     \bigg( \sum _{a\in S} a\bigg)^m c  =   \sum_{a\in S} a^mc,$$   for  $S \subset \tM$ finite.

Note that when $R$ is divisible, these relations are formal
consequences of \eqref{eq:Frobenius} since writing $c = d^m$ we
have $$d^m + \bigg( \sum _{a\in S} a\bigg)^m   = \bigg(d+
\sum_{a\in S} a\bigg)^m = d^m +   \sum _{a\in S} a^m  ; \qquad d^m
\bigg( \sum _{a\in S} a\bigg)^m   = \bigg(d  \sum_{a\in S}
a\bigg)^m = d^m \sum _{a\in S} a^m.$$

The Frobenius monoid \semiring0 $\httM$  satisfies the following
universal property:

 \begin{prop}\label{FrobUn} Suppose $\vrp: \tM \to \tM'$ is a
    monoid homomorphism,  where  the  monoid $\tM' $ is
    totally ordered. Viewing $\tM' $ as a bipotent \semiring0 via
    Proposition~\ref{maxplus1}, we have a
    natural homomorphism $$\htvrp: \httM \to \tM'$$ given by $$\htvrp
     \bigg(\sum _{a \in S} a \bigg) = \sum_{a \in S}  \vrp (a),$$ satisfying the universal property that    $\vrp$
    factors as
\begin{equation*}
    \xymatrix{
    \vrp : \tM    \ar[r] & \httM  \ar[r]^{\htvrp} & \tM' . \\
}
\end{equation*}
\end{prop}
\begin{proof} The map   given by  $ \sum _{a \in S} a  \mapsto \sum_{a \in S}  \vrp (a)$
 is  the desired \semiring0 homomorphism, since $\tM'$
satisfies the Frobenius property.
 \end{proof}

On the other hand, the same argument shows that $\httM$  itself is
not ordered as a monoid, since one can provide any ordered monoid
$\tM$ with two orders, one order making $a+b = a$ and the reverse  order
making $a+b =b,$ as shown in Remark~\ref{rmk:duality}. The point
here is that the Frobenius property, being an algebra identity,
permits the definition of a universal. Furthermore, one could
define the category of \semirings0 satisfying the Frobenius
property, which comprises a full subcategory of $\SeRo.$ Although
Proposition~\ref{FrobUn} indicates that this is the ``correct''
category in which to conduct much of the investigation in tropical
algebraic geometry, we forego further consideration of this
category in this paper.

\section{Integral domains and monoids with
valuation}\label{MonVal}

We turn to the main notion of ``tropicalization.'' As indicated in
the introduction, we need to consider integral domains $W$ with
valuation $v: W\setminus \{ \zero_W \} \to \tG, $ having cancellative (ordered) value monoid $\tG$. In
valuation theory it is customary to write the operation of the
value monoid $\tG$ as addition, and to utilize the axiom
\begin{equation*} v(a+b) \ge \min \{ v(a), v(b) \}.\end{equation*}
 Note that
we can replace the valuation $v$ by $v' :=-v$ to get the dual
equation
\begin{equation}\label{max} v'(a+b) \le \max \{ v'(a), v'(b) \}.\end{equation}

We adjust the notation of valuation theory to fit in with the
algebraic language of \semirings0. Thus, from now on in this
paper, we use multiplicative notation, written $\tG := (\tG, \cdot
\, , \geq, \one_\tG)$, for the value monoid~$\tG$ with unit
element $\one_\tG$, which can be viewed as a \semiring0 via
Proposition~\ref{catid1}, and use~\eqref{max} for the valuation
axiom, since it fits in better with the \semiring0 approach. (But
  several authors, such as
 Sturmfels and his school, have used the min-plus algebra instead,
 in order to forego taking the negative.)


\begin{defn}\label{eq:PowerSeries0} The algebra of
\emph{Puiseux series} $\mathbb K$ over an algebraically closed
field $K$
  is the field  of series of the
form
\begin{equation}\label{eq:PowerSeries}
 p :=
\sum_{\tau \in T} c_{\tau} t ^{\tau},\qquad c_\tau \in K,
\end{equation} with $T \subset \Real $  well-ordered  (from below).
\end{defn}

Sometimes one takes $T \subset \Rati;$ any totally ordered field
will do. We will take $\Real $ in this paper.  For any field $F$, we write  $\Fld ^\times $ for  $\Fld  \setminus \{ \zero\}.$ The tropical
connection is that the max-plus algebra appears as the target of
the valuation $$ \nVal : \Fld ^\times \To
 \Real   \ $$ given by sending $p(t)\ne
      \zero_{\Fld}$ to the negative of the lowest  exponent of its  monomials having nonzero coefficient;
 \begin{equation}\label{eq:valPowerSeries}
  \nVal(p)\  :=
    - \min \{\tau \in T \ : \; c_{\tau} \neq \zero_K \}
     .
\end{equation}

\begin{rem}\label{terminology0} There is a natural multiplicative map $\pi: \mathbb K^\times \to
K^\times$, sending a Puiseux series $
 p=
\sum_{\tau \in T} c_{\tau} t ^{\tau}$ to $c_{\nVal(p)}$.

This gives an extra important piece of information, since for any
two Puiseux series $p,q$ we must have $ \nVal (p+q) = \max \{
\nVal (p),  \nVal(q) \}$ unless $ \nVal(p) =  \nVal(q)$ and
$\pi(p) = \pi (q)$, in which case $ \nVal (p+q) <  \max \{ \nVal
(p),  \nVal(q) \}$. In this way, $\pi$ measures how much bipotence
is lost with respect to  $\nVal.$
\end{rem}

\subsection{Valued monoids}\label{ssec:Mono}

We view the previous observations in a somewhat more general
setting.

\begin{defn}\label{def:valuedMonoid} A monoid $\tM = (\tM,\cdot \,, \mone)$
is \textbf{m-valued} with respect to a totally ordered
monoid~$\tG: = (\tG,\cdot \, , \geq, \one_\tG)$  if there is an
onto monoid homomorphism $v : \tM \to \tG$.
This set-up is    notated as the \textbf{triple}~$(\tM,\tG,v)$.
\end{defn}

 \begin{note} The hypothesis that $v$ is onto can always be
attained by replacing $\tG$ by $v(\tM)$ if necessary. \end{note}
Given a field
 with valuation, or more generally, an integral domain $W$ with valuation
 $v: W\setminus \{ \zero_W \} \To \tG ,$
  we take  $\tM =  \setminus \{ \zero_\tM \} $, a~cancellative submonoid of
$W$,
 to obtain
the 
triple $(\tM, \tG, v)$ as in Definition \ref{def:valuedMonoid}.
When $W$ is an arbitrary commutative ring with valuation, we must
assume that the monoid $\tG$ is pointed, and take the triple $(W,
\tG, v).$

\begin{example}\label{Pass111} Another major example of an m-valued monoid is $\tM =
(\mathbb C ^ \times , \cdot \, , 1 )   $, $\tG = (\Real_{\geq
0},+, \geq, 0)$, and $v : \mathbb C ^ \times  \to \Real_{\geq 0}$,
given by $v: z \mapsto \log _t (|z|),$ where $t$ is a given
positive parameter. This leads us to the theory of complex
amoebas, cf.~Passare \cite{Pass}.\end{example}

The category of m-valued monoids is quite intricate, since the
morphisms should include all maps which ``transmit'' one
m-valuation to another, as defined in \cite{KZ1}. In order to
simplify this aspect of the theory, we restrict ourselves to a
subcategory, but consider the general version in \cite{IKR6}.

 \begin{defn}
 $\Valmon$ is the category of m-valued monoids whose objects are triples $(\tM,
\tG, v)$ as in Definition~\ref{def:valuedMonoid}, for which a
morphism \begin{equation}\label{eq:valMonMor} \phi : (\tM, \tG, v)
\To (\tM', \tG', v')\end{equation} is comprised of a pair
$(\phi_\tM, \phi_\tG)$ of a monoid homomorphism $\phi_\tM: \tM \to
\tM'$, as well as an order-preserving monoid homomorphism
$\phi_\tG: \tG \to \tG'$, satisfying the compatibility condition
\begin{equation}\label{comm1} v'(\phi_\tM (a)) = \phi_\tG (v(a)),
\quad \forall a \in \tM.
\end{equation}

$\SValmon$ is the full subcategory of  $\Valmon$ in which the
target monoid $\tG$ is cancellative.
 \end{defn}

Thus, we have the  categories $\Valfld$ (resp.~$\Valdom$) whose
objects are fields (resp.~ integral domains) with
 valuations to cancellative monoids, and whose morphisms are ring
homomorphisms which restrict to morphisms in $ \SValmon$,%
and each has its respectful forgetful functor to $\SValmon$.


\begin{rem}\label{6.2} If
$(\tM,\tG,v)$ is a triple, then $v$ induces a pre-order $\le $ on
$\tM,$ given by $a \le b$ iff $v(a) \le v(b)$ in $\tG.$\end{rem}

\begin{lem} There is a fully  faithful   functor of categories $\Valmon \to
\PreOMon.$\end{lem}
\begin{proof} The functor is given by Remark~\ref{6.2}.

Conversely, given any monoid $\tM$ with pre-order, we define $\tG
:= \deqv{\tM}$ as in Proposition~\ref{equiv1}. Then we define $v:
\tM\to \tG$ by $a\mapsto [a];$ clearly $(\tM, \tG, v)$ is a
triple. One sees easily that the morphisms match.\end{proof}

On the other hand, as we have observed, the main idea of
tropicalization is the following observation:

\begin{rem}\label{6.20}
The valuation itself provides a forgetful functor
$\tF_{\operatorname{val}}:\SValmon \to \OMon^+$, where we remember
only the target monoid $\tG$ from the triple $(\tM, \tG, v)$ .
\end{rem}

\begin{rem}\label{terminology} Let us recall some valuation theory,
which we can state in terms of an integral domain~$W$ with
valuation $(W,\tG,v)$.
 The \textbf{valuation
 ring} $R$ (resp.~\textbf{valuation
 ideal} $\mfp$)   is the set of elements of $W$ having value $\ge 0$
 (resp.~ $>0$); the \textbf{residue
domain}
$\overline{W}$ is $R/\mfp.$
\end{rem}
%
%

The residue  domain   is  an integral domain. When $W$ is a field $F$, the residue domain $\bar F$ is also a field.
For example, the valuation ideal $\mfp$ of $\mathbb K$ of Definition~\ref{eq:PowerSeries0} is the set of
Puiseux series having value
 $>0$, and the residue field can be identified with $K$.

 Here is another example, to illustrate some subtler aspects of the
definitions.
\begin{example}\label{ex:K1}$ $
\begin{enumerate}\eroman

\item The integral domain $\C [\la _1, \la _2]$ has the natural
valuation to $\Z \times \Z,$ ordered via the lexicographic order
where $v(\la_1) = (1,0)$ and $v(\la_2) = (0,1).$ On the other
hand, there is the valuation $\tlv: \C [\la _1, \la _2] \to \Z$
given by $\tlv(\sum _{i,j} \a_{i,j}\la _1^i \la_2 ^j) = k$ for
that smallest $k = i+j$ such that $\a_{i,j}\ne 0.$  (In other
words, $\tlv(\la_1) = \tlv(\la _2) = 1.$) The identity map
$(1_{\C[\la _1, \la _2]},1_{ \Z \times \Z})$ (where we replace the
valuation $v$ by $\tlv$) is not a morphism in $ \SValmon$ since it
is not order-preserving. $v(\la_1) = (1,0) > (0,2) = v(\la_2^2)$
whereas $\tlv(\la_1) = 1< 2  = \tlv(\la_2^2).$ \pSkip

\item In (i), we take a different order on   $\Z \times \Z,$ where
two pairs are  ordered first by the sum of their coordinates and
then only secondarily via the lexicographic order. Now the
identity map (where we replace the valuation $v$ by $\tlv$) is a
morphism in $ \SValmon$ since it is order-preserving. Note however
that it is not strictly order-preserving, since $v(\la_1) = (1,0)
> (0,1) = v(\la_2)$ whereas $\tlv(\la_1) = 1 = \tlv(\la_2).$
\end{enumerate}
\end{example}

%
%

%
%

\section{The layered structure}\label{Sec:SupCat}

We are ready to bring our leading player. In this section we
describe the algebraic category in whose context we may formulate
all the algebraic structure we need (including matrices and
polynomials) for the
  layered theory.  To
simplify notation and avoid technical complications, we work with
a \semiring0~$L$ without a zero element, even though information
is lost; the full theory is given in~\cite{IKR6}. Much of the
layered theory stems from the following fundamental construction
from \cite{IzhakianKnebuschRowen2009Refined}, which is inspired by
\cite{AGG}.

\begin{construction}\label{defn50}  $R := \scrR(L,\tG)$ is   defined
 set-theoretically as $L \times \tG $, where
 we denote the element $(\ell,a)$ as $\xl{a}{\ell}$  and, for
$ k,\ell\in L,$ $a,b\in\tG,$ we define multiplication
componentwise, i.e.,
\begin{equation}\label{13}   \xl{a}{k} \cdot \xl{b}{\ell} =
\xl{ab}{k\ell}.
\end{equation}

Addition is given by the usual rules:

\begin{equation}\label{14}
 \xl{a}{k} + \xl{b}{\ell}=\begin{cases}  \xl{a}{k}& \quad\text{if }\ a >
 b,\\ \xl{b}{\ell}& \quad\text{if }\ a <  b,\\
 \xl{a}{k+\ell}& \quad\text{if }\ a= b.\end{cases}\end{equation}

 We define
  $R_\ell : = \{ \ell \} \times \tG  $, for each $\ell \in L.$
 Namely $R = \dot \bigcup_{\ell \in L } R_\ell$.
\end{construction}

This is to be our prototype of a layered \bidomain0, and should be
borne in mind throughout the sequel.
 Nevertheless, one
 should also consider the possibility that the monoid $\tG$ is non-cancellative, in which case, as noted in \cite{IzhakianKnebuschRowen2009Refined},
 Construction~\ref{defn50} fails to satisfy distributivity and thus is not a \semiring0. This difficulty can be resolved,
but the ensuing category becomes rather technical, so we defer it
to
 \cite{IKR6}.

\subsection{Layered pre-\domains0}

We axiomatize  Construction~\ref{defn50} in order to place it in
its categorical framework.

\begin{defn}\label{defn10} Suppose  $(\Lz, \ge)$ is a partially pre-ordered
\semiring0. An $\Lz$-\textbf{layered pre-\domain0} $$\R :=\ldR,
\qquad
$$ is a \semiring0 $\R $, together with a partition $\{ R_\ell :
\ell \in L\}$ into disjoint subsets $R_\ell \subset R$, called
\textbf{layers}, such that
\begin{equation}\label{unionp} \R := \dot \bigcup_{\ell\in \Lz}\R
_\ell,\end{equation}  and a family
  of \textbf{sort transition maps}
$$   \nu_{m,\ell}:\R _\ell\to \R _m,\quad
\forall m\ge \ell >0 ,$$  such that $$\nu_{\ell,\ell}=\id_{\R
_\ell}$$ for every $\ell\in \Lz,$ and
$$\nu_{m,\ell}\circ \nu_{\ell,k}=\nu_{m,k} , \qquad  \forall m \geq \ell \geq k, $$ whenever both sides
are defined,  satisfying the  following axioms A1--A4,  and B.

We say that any element $a$ of $\R _k$ has \bfem{layer}~$k$ $(k\in
\Lz)$. We write $a \nucong b$ for $b \in R_\ell$,  whenever
 $\nu_{m,k}(a) = \nu_{m,\ell}( b )$ in $R_m$ for some $m \ge
 k,\ell$. (This notation is used generically: we write  $a \nucong b$
even when the sort transition maps  $\nu_{m,\ell}$ are notated
differently.)

 Similarly, in line with Remark \ref{maxplus}, we write $a \le _\nu b$ if $\nu_{m,k}(a) + \nu_{m,\ell}( b )=  \nu_{m,\ell}( b )$ in $R_m$ for some $m \ge
 k,\ell$.

 The axioms are as follows:

\boxtext{
\begin{enumerate}

\item[A1.] $\rone \in \R _{1}.$ \pSkip

 \item[A2.] If $a\in \R _k$ and  $b\in
\R _\ell,$ then $ab\in  \R_{k \ell} $. \pSkip

\item[A3.] The product in $\R $ is compatible with sort transition
maps: Suppose $a\in \R _k$ and $b\in \R _{\ell},$ with $m\ge k$
and $m'\ge \ell.$ Then

$\nu_{m,k}(a)\cdot\nu_{m',\ell}(b)= \nu_{mm',k\ell}(ab).$

\item[A4.] $\nu_{\ell,k}(a) + \nu_{\ell',k}(a)
 =\nu_{\ell+\ell',k}(a)  $ for all $a \in R_k$ and all $\ell, \ell' \ge k.$ \end{enumerate}}

 \boxtext{
\begin{enumerate}

\item[B.] (Supertropicality) Suppose  $a\in \R _k,$ $b\in \R
_{\ell},$ and $a \nucong b$. Then \\ $a+b \in R_{k+\ell}$ with
$a+b \nucong a$.\\  If moreover $k= \infty,$ then $a+b = a.$
\end{enumerate}}
%

$L$ is called the \textbf{sorting \semiring0} of the $L$-layered
pre-\domain0 $R = \bigcup_{\ell \in L} R_\ell$.



 \end{defn}

For convenience, we assume in the sequel that $L = L_{\ge 0},$
i.e., all nonzero elements of $L$ are positive. Often $L $ is
$\Net$ or $\Net^+$.

\begin{rem}\label{tang1}
 The $L$-layered pre-\domain0 \ $R$ has
the special layer $R_1$, which is a multiplicative monoid, called
the monoid of \textbf{tangible elements}, and acts with the
obvious monoid action (given by multiplication) on each layer
$R_k$ of $R$.
 \end{rem}

Thus, in one sense, $R$ extends its monoid  of tangible elements.
Although  we have given up bipotence, and Axiom B provides us the
slightly weaker notion of \textbf{$\nu$-bipotence}, which says
that $a +b \nucong a$ or $a +b \nucong b$ for all $a,b \in R.$

\begin{defn}
An $\Lz$-\textbf{layered \bidomain0} is a $\nu$-bipotent
$\Lz$-layered  pre-\domain0 which is (multiplicatively)
cancellative. (We use the prefix ``bi'' in this paper to stress
the  $\nu$-bipotence.)

%

 An $L$-layered \bidomain0 $R$ is called an
\textbf{$L$-layered \bisemifield0} if $(R_1,\cdot \, )$ is an
Abelian  group.
\end{defn}

Note that according to this definition, an
 $L$-layered \bisemifield0 need not be a  \semifield0 unless  $L$ itself also is  a
multiplicative group (and thus a \semifield0).
When $R$ is an $L$-layered \bisemifield0, the action of
Remark~\ref{tang1} is simply transitive, in the sense that for any
$a,b \in R_\ell$ there is  a unique element $r \in R_1$ for which
$ar = b.$



\begin{defn}\label{subsemi} We write $R_{>\ell}$ (resp.~ $R_{\ge \ell}$ )
 for $\bigcup _{k > \ell}\,  R_k$ (resp.~$\bigcup _{k \ge \ell}\,  R_k$). \end{defn}

  \begin{defn}\label{ghostsurp} The
  layer $R_1$ of an $L$-layered \predomain0
  $R$ is  called the \textbf{monoid of tangible elements}. \end{defn} We are interested in the case that $R_1$ generates $R$.

\begin{lem}\label{gen1} If $\tM$ is any submonoid of a layered  pre-\domain0  $\R
:=\ldR,$ then the additive sub-semigroup $\overline{\tM}$ of $R$
generated by $\tM$ is also a layered pre-\domain0.\end{lem}
 \begin{proof} $\overline{\tM}$ is a \semiring0, by distributivity. Axiom  A1 is given, and the other axioms follow
 a fortiori.
 \end{proof}

\begin{defn}\label{tangen}
The \textbf{tangibly generated} sub-\semiring0 is the
sub-\semiring0 generated by $R_1$; if this is~$R$, we say that $R$
is \textbf{tangibly generated}.
 \end{defn}
%
%
%

\begin{rem}\label{Krems} Several initial observations are in order.
\begin{enumerate} \eroman

\item The layered structure resembles that of a graded algebra,
with two major differences: On the one hand, the condition that
$R$ is the disjoint union of its layers is considerably stronger
than the usual condition that $R$ is the direct sum of its
components; on the other hand, Axioms ~A4 and~B show that the
layers are not quite closed under addition. \pSkip

\item This paper is mostly about $L$-layered \bidomains0  (in
particular, $L$-layered \bisemifields0). However, since
$\nu$-bipotence does not hold for polynomials, one considers the
more general $L$-layered pre-\domains0  when studying polynomial
\semirings0. \pSkip

\item  For each $\ell\in \Lz$ we introduce the sets
$$\FR_{\ge \ell}:= \bigcup_{m\ge \ell}\R _m \qquad{and} \qquad \FR_{>\ell}:=
\bigcup_{m> \ell}\R _m .$$

Many of  our current examples satisfy $L = L_{\ge 1}$ and thus $R
= \FR_{\ge 1} $. When $R$ is an $L$-layered \bidomain0, we claim
that $\FR_{\ge 1}$ is an $L_{\ge 1}$-layered sub-\bidomain0\ of
$R$, and $\FR_{\ge k }$ and $\FR_{>k}$ are \semiring0 ideals
of~$\FR_{\ge 1}$, for each $k \in L_{\ge 1}.$ Indeed, this is an
easy verification of the axioms, mostly  from Axiom~A2. \pSkip

 \item Given any $L$-layered \bidomain0 $R$ and any multiplicative submonoid $\tM$ of
 $R_{\ge 1}$, we want to define the $L$-layered sub-\bidomain0 of $R$
 generated by $\tM$. First we take $$\tM' := \{ \nu_{\ell, k}(a): \ell \ge k \in L, \ a \in
 \tM \cap R_k\},$$ which is a submonoid closed under the transition maps.
Then we take $$\tM'': = \tM' \cup \{ a+b: \ a,b \in \tM' \quad
\text{with} \quad
 a\nucong b \}.$$ This is closed under all the relevant operations,
so is the desired $L$-layered \bidomain0. Note that the second
stage is unnecessary for $a=b$, in view of Axiom~A4. \pSkip

\item Although ubiquitous in the definition, the sort transition
maps get in the way of computations, and it is convenient to
define the elements
  \begin{equation}
\label{idemp} e_\ell:= \nu_{\ell,1}(\rone)\quad (\ell\ge 1).
\end{equation}
If $a\in \R _k,$\ $\ell\in \Lz,$ and $\ell\ge 1,$ we conclude by
Axiom A3 that
$$\nu_{ \ell\cdot k,k }(a)=\nu_{\ell\cdot k,1 \cdot k}(a\cdot \rone )
= \nu_{\ell,1}(\rone)\cdot\nu_{k,k}(a)=\nu_{\ell,1}(\rone)\cdot a
= e_\ell a.$$ Thus the sort transition map $\nu_{ \ell\cdot k,k }$
means multiplication by $ e_\ell.$

Note that $e_k + e_\ell = e_{k+\ell}$ by Axiom A4.

 The element
$e_\ell$ is a (multiplicative) idempotent of $R$ iff $\ell^2 =
\ell$ in $L$. In particular, $e_1$ and $e_\infty$ (when $\infty
\in L$) are idempotents of $R$.

\end{enumerate}
\end{rem}

Let us introduce the \textbf{sorting map} $s :\R\to L,$ which
sends every element $a\in\R_\ell$ to its sort index $\ell$, and we
view the \semiring0  $\R$ as an object fibered by~$s$ over the
sorting \semiring0~$L$.

\begin{rem}  Axioms A1 and A2   yield
the conditions
\begin{equation}\label{sortval} \lv(\rone)=\lone, \qquad \lv(ab)=\lv(a)\lv(b), \qquad \forall a,b\in\R.\end{equation}
Also,  Axiom A4 yields $s(a+a) = s(a) + s(a) = 2s(a),$ thereby
motivating us to view addition of an element with itself as
doubling the layer. Applying $\nu$-bipotence to Axiom B  shows
that
$$s(a+b) \in \{ s(a), s(b) , s(a)+s(b)\}.$$
\end{rem}

 To emphasize the sorting map, as well as the order on $L$, we sometimes write
$\ldsPR$ for a given $\Lz$-layered \bidomain0 $\R$ with sort
transition maps $(\nu_{m,\ell}: m \ge \ell)$ and their
accompanying sorting map $s: R\to L.$

%
%
%
%

\subsection{Uniform $L$-Layered \bidomains0 }\label{sec:1to1}$ $

There are two main  examples coming from tropical mathematics.

\begin{enumerate} \ealph \item Let $R = \scrR(L,\tG)$ (corresponding to the ``naive'' tropical geometry). By Construction~\ref{defn50}, the sort transition maps $\nu_{m,\ell}$ are
all bijective. \pSkip

\item
Suppose $\mathbb K$ is the field of Puiseux series $\{ f: = \sum _{u \in
\Q} \a _u \la ^u :$ $f$ has well-ordered support$ \}$ over a
given field $F$. Then we have the $m$-valuation $v: \mathbb K \to \Q$
taking any Puiseux series $f$ to the lowest real number $u$ in its
support.

We incorporate
$\mathbb K$ into the structure of $R$, by putting $R_\ell$ to be a
copy of $\mathbb K$ for $\ell \le 1$ and $R_\ell$ to be a copy of
$\tG$ for  $\ell
> 1$. We take the $\nu_{m,\ell}$ to be $v$ whenever $m > 1  \geq
\ell$, and the identity otherwise. In this way, $v$ is preserved within the structure of $R$.
\end{enumerate}

We focus on the first case, since one can reduce to it anyway via
the equivalence given below in Definition~\ref{equivrel} (which
takes us from the usual algebraic world to the tropical world).

\begin{defn} An $L$-layered pre-\domain0 $R$ is  \textbf{uniform} if
$\nu_{\ell,k}$ is 1:1 and onto for each $\ell >k.$

\end{defn}
\begin{example} Example~\ref{defn50} is a uniform $\Lz$-layered
\bidomain0, when the monoid $\tG$ is cancellative.
\end{example}

Let us see how the layered theory simplifies for uniform
$L$-layered \bidomains0, enabling us to remove the sort transition
maps $\nu_{\ell,k}$ from the picture.

\begin{lem}\label{tang11}
Any element $a \in R_\ell$ can be written uniquely as $e_\ell a_1
= \nu _{\ell,1} (a_1)$ for $a_1 \in R_1.$
\end{lem}
\begin{proof} Existence   and uniqueness of $a_1$ are clear since $\nu _{\ell,1}$
is presumed to be~1:1. The last assertion follows from Axiom A3.
\end{proof}

\begin{prop}\label{equalnu}  In a uniform $\Lz$-layered \bidomain0, if $a
 \nucong b$ for $a \in  R_k$ and $b\in R_\ell$ with $\ell \geq k$ then $b = \nu_{\ell,k}(a).$ In particular, if $a
 \nucong b$ for $a,b \in R_\ell$, then $a=b.$
\end{prop}
\begin{proof}  An immediate application of Lemma~\ref{tang11}.
\end{proof}

Now we can remove the sort transition maps from the definition,
when we write $R = \bigcup_{\ell \in L}\ e_\ell R_1.$

\begin{prop}\label{removesort} If $(L,\cdot \, )$ is a
multiplicative group, then one has $\nu_{\ell,k}(a_k) \nucong e_m
a_k$ where $\ell = m k$.

In a uniform $\Lz$-layered \bidomain0 (for $L$ arbitrary), the
transition map $\nu_{m,\ell}$ is given by $e_\ell a_1 \mapsto e_m
a_1$.

\end{prop}
\begin{proof} If $\ell = k$ there is nothing to prove, so we assume that $\ell > k $ are non-negative, and write $\ell = k + p$ for $p \in L.$ Then $\ell =
k(1 + pk^{-1}),$ and $m = 1 + pk^{-1}.$ Now  $e_m a_k \in R_\ell$,
and $$e_m a_k \nucong a_k \nucong \nu_{\ell,k}(a_k).$$ The second
assertion now is clear.\end{proof}

Thus the sort transition maps have been replaced by multiplication
by the $ e_m.$
 Note that $\nu$-bipotence and Axiom B$'$ could then be used as the definition for
addition in $R$, and we summarize our reductions:

\begin{prop}\label{unif1} A uniform $\Lz$-layered \bidomain0 can be described as
the \semiring0 $$R := \dot\bigcup _{\ell \in L} R_\ell,$$ where
 each $R_\ell = e_\ell R_1$, $(R_1,\cdot \, )$ is a monoid, and there is a 1:1 correspondence
 $R_1 \to R_\ell$ given by $a \mapsto e_\ell a$ for each $a \in R_1.$
\end{prop}

\begin{prop}\label{tang111} In the uniform case, axioms A2 and A3 can be replaced by the
respective axioms:

\boxtext{
\begin{enumerate}
 \item[A2$'$.] If $a = e_k a_1\in \R _k$ and  $b = e_\ell b_1\in
\R _\ell,$ for $a_1,b_1 \in R_1,$ then $ab = (a_1b_1)e_{k\ell}.$
 \item[A3$'$.] $e_\ell e_k = e_{\ell k}$ for all $k, \ell \in L.$

\end{enumerate}}

  Furthermore,  Axiom A4 now is equivalent
to  Axiom B, which we can reformulate as: \boxtext{
\begin{enumerate}
    \item[B$'$.] If $a  = e_k a_1$ and $b = e_\ell a_1$ (so that $ a \nucong b$), then $a+b =
e_{k+\ell} a_1.$
\end{enumerate}}

The operations in $R$ are given by Axioms A2$\,'$, A3$\,'$,
B$\,'$, and $\nu$-bipotence.\end{prop}

\begin{proof} Axiom A3 follows from the observation that $e_\ell a_1
e_k b_1 = e_{\ell k} (a_1b_1)$; when $a_1, b_1 \in R_1$ then
$a_1b_1 \in R_1.$
\end{proof}



\subsection{Reduction to the uniform case}

In one sense, we can reduce the general case of an $L$-layered
pre-\domain0\ ~ $R$ to the uniform case. First we cut down on
superfluous elements. Note that if $\nu_{k,1}$ are onto for all $k
\ge 1$, then all the $\nu_{\ell,k}$ are onto for all $\ell \ge k$.
Indeed, if $a \in R_\ell$ then writing $a = \nu_{\ell,1}(a_1)$ we
have
$$a = \nu_{\ell,k}(\nu_{ k,1}(a_1)).$$

\begin{rem} Suppose $L = L _{\ge 1}.$ For any $L$-layered pre-\domain0 $\R := \ldR,$
if we replace $R_k $ by $\nu_{ k ,1}(R_1)$ for each $k \in L$, we
get an $L$-layered \bidomain0 for which all the $\nu_{\ell,k}$ are
onto. \end{rem}

Having reduced many situations to the case for which all the
$\nu_{\ell,k}$ are onto, we can get a uniform $L$-layered
pre-\domain0 by specifying when two elements are
``interchangeable'' in the algebraic structure.

\begin{defn}\label{equivrel} Define the equivalence relation
$$a \equiv b \quad \text{ when } \quad  s(a) = s(b) \  \text{and} \   a\nucong b.$$\end{defn}

In view of Proposition~\ref{equalnu}, this relation is trivial in
case $R$ is a uniform $L$-layered \bidomain0.

\begin{prop}\label{equivlem1}  The binary relation $<_\nu$ on an $L$-layered
pre-\domain0 $R$ induces a pre-order on the \semiring0 of
equivalence classes $\deqv{R}$. Furthermore, if $a \equiv b$, then
$ac \equiv bc$ and $a+c \equiv b+c$ for all $c\in R.$ Thus,
$\equiv$ is a congruence. \end{prop}
\begin{proof} The first assertion is immediate. For the second
assertion,  $s(ac) = s(a)s(c) = s(b)s(c) = s(bc)$ and $ac\nucong
bc$, proving $ac \equiv bc$.

Next, we consider addition. If $a
>_\nu c$, then $$a+c = a \equiv b = b+c.$$ If $a <_\nu c$, then $a+c = c = b+c.$ If  $a \nucong
c$, then $$s(a+c) = s(a) + s(c) =  s(b) + s(c) = s(b+c),$$ and
$a+c\nucong a  \nucong b \nucong b+c$.
\end{proof}

Let us summarize.

\begin{cor}\label{11red}  When the transition maps  $\nu_{ \ell,k}$ are onto, one can reduce to uniform $L$-layered \bidomains0, by means of the equivalence
relation $\equiv$ of Definition~\ref{equivrel}.\end{cor}
\begin{proof} Any
$\nu$-equivalent elements having the same sort are identified.
Then Proposition~\ref{equivlem1} shows that $\deqv{R}$ is an
$L$-layered \bidomain0, under the natural induced layering, and
the transition maps on $\deqv{R}$ clearly are bijective.
\end{proof}



 \subsubsection{Ghosts and the surpassing relation}\label{ghocom1}

We want a layered version of ghosts.

\begin{defn}\label{ghostsurp1} An $\ell$-\textbf{ghost sort} is an element $\ell + k \in L$ for some $0 \ne k \in
 L$.  An element   $b \in R$ is an
$\ell$\textbf{-ghost} (for given $\ell \in L$) if $s(b)\in L$ is
an $\ell$-ghost sort.   A \textbf{ghost element} of $R$ is a
1-ghost.
\end{defn}

Thus, the relation
 $(\ge)$ on $L$ satisfies
$$m \ge \ell \quad \text { when }\ \left\{
\begin{array}{l}
  m = \ell   \\
\text{or} \\
  \text{$m$ is an $\ell $-ghost sort.}
\end{array}\right.
$$
%

Here is a key relation in the theory, even though it does not play
a major role in this discussion.

\begin{defn}\label{ghostsurpL}   The \textbf{$L$-surpassing relation}
$\lmodWL$ is given by
 \begin{equation}  a \lmodWL b \quad
\text{  iff either } \quad  \begin{cases} a=b+c &  \
\text{with}\ c \ \text{an} \  s(b)\text{-ghost},
                 \quad \\ a=b,\\   
   a\nucong b & \ \text{with}\
 a  \
\text{an} \  s(b)\text{-ghost}.\end{cases} \end{equation}
\end{defn} It follows that if $a \lmodWL b$, then $ a+b$ is
$s(b)$-ghost. When $a \ne b$, this means $a \ge _\nu b$ and $a$ is
an $s(b)$-ghost.

\begin{rem}\label{ghsurp}  If $a \lmodL b$, then clearly $a +c \lmodL b+c$ and $a c \lmodL
bc$. Thus $\lmodL$ respects the \semiring0 operations.
\end{rem}

\begin{rem}\label{Fro0} If $a >_ \nu  b$, then $(a+b)^m = a^m $. Hence, the Frobenius property $(a+b)^m = a^m + b^m$ is satisfied in an
$L$-layered \predomain0 whenever $a \not \nucong b$.
We always have $(a+b)^m \lmodL  a^m + b^m$.
\end{rem}
\subsection{Layered homomorphisms}\label{ghocom}

In line with the philosophy of this paper, we would like to
introduce the category of $L$-layered \predomains0. This entails
finding the correct definition of morphism. We start with the
natural definition from the context of \domains0. Although this definition is good enough for the
purposes of this paper, a more sophisticated analysis would require us
to to consider the notion of ``supervaluation'' from \cite{IzhakianKnebuschRowen2009Valuation}, and how this
relates to morphisms that preserve the layers.
%
Here we take the morphisms in this category to be \semiring0
homomorphisms which respect the order on the sorting  semiring
$L$:

\begin{defn}\label{gsmorph} A \textbf{layered homomorphism}
of  $L$-layered \predomains0 is a map
\begin{equation}\label{eq:layMor} \Phi:=(\vrp,\rho): \RLsnu
\to\RLsnuT \end{equation} such that $\rho:L \to L' $ is a
\semiring0
 homomorphism,
 together with a \semiring0
 homomorphism $\vrp: R \to R'$ such that
\boxtext{
\begin{enumerate}
   \item[M1.] If $\vrp(a) \notin R_0',$ then  $\lv '(\vrp
 (a)) \ge \rho (\lv(a)).$ \pSkip

  \item[M2.]   If  $a \cong
_ {\nu}b$,
 then  $\vrp(a)  \cong _ {\nu} \vrp(b). $
(This is taken in the context of the ${\nu}'_{m',\ell'}$.)
%
\end{enumerate} }

\end{defn}

%

We always denote $\Phi = (\vrp, \rho )$
 as $\Phi: R \to R'$ when unambiguous. In most of the
following examples, the sorting \semirings0 $L$ and~$L'$ are the
same. Accordingly, we call the layered homomorphism $\Phi$ an
$L$-\textbf{homomorphism} when $L = L'$ and $\rho = 1_L.$

%
%

\begin{prop}\label{strt}Any layered
homomorphism $\varphi$ preserves $\nu$, in the following sense:

  If $a \ge _\nu b,$   then $\varphi(a) \ge _{\nu'}
\varphi(b).$ \end{prop}
\begin{proof} $ $
  $\varphi(a) \cong_{\nu'} \varphi(a+b) = \varphi(a)+ \varphi(b),$ implying
$\varphi(a) \ge _{\nu'} \varphi(b).$
\end{proof}

\begin{prop}\label{tangen1} Suppose   $\varphi:  R  \to R'$ is a layered
homomorphism, and  $R$ is tangibly generated (cf.~
Definition~\ref{tangen}). Then $\varphi$ is determined by its restriction to $R_1$, via the formula
$$ \varphi (a+b) =  \varphi(  a) + \varphi(  b), \quad \forall a,b \in R_1.$$\end{prop}
\begin{proof} It is enough to check sums, in view of
Lemma \ref{gen1}. We get the action of $\varphi$ on all of $R$ since $R_1$  generates $R$.
\end{proof}

  \begin{rem} The definition given here of layered homomorphism is
 too strict for some applications. One can weaken the definition of layered
homomorphism  by utilizing the surpassing relation, requiring
merely that $ \varphi(a)+ \varphi(b)  \lmodL \varphi(a+b),$
 but various technical difficulties arise, so we defer the study of this
category to \cite{IKR6}.
\end{rem}

Before continuing, let us see how this definition encompasses
various prior tropical situations.

\begin{example}\label{exmp:layredStr} We assume throughout that
$R$ is an $L$-layered pre-\domain0.
\begin{enumerate}
\eroman   \ddispace \item In the max-plus situation, when $L = \{
1 \},$ $\rho$ must be the identity, and $\Phi$ is just a
\semiring0 homomorphism. \pSkip

\item In the ``standard supertropical situation,'' when $L = \{ 1,
\infty \},$ $\Phi$~must send the ghost layer $R_\infty$ to~
$R_\infty$. If $\mfa \triangleleft R,$ one could take $R'_1 :=
R\setminus \mfa$ and $R'_\infty :=  R_\infty \cup \mfa$. The
identity map is clearly a layered homomorphism; its application
``expands the ghost ideal'' to~$\mfa,$ thereby taking the place of a
 semiring homomorphism to the factor \semiring0. \pSkip

%

\item Generalizing (ii), we obtain layered homomorphisms by
modifying the layering. We say a \textbf{resorting map} of a
uniform $L$-layered pre-\domain0 $R$ is a map $s': R \to L$
satisfying the following properties:
\medskip

\begin{enumerate} \eroman \item $s'(\one_R) = 1,$
\pSkip

 \item $s'(R_1) \subseteq L_{\ge 1},$ \pSkip

\item $s'(ab) = s'(a)s'(b),   \ \forall a,b \in R_1,$  \pSkip

\item $s'(e_\ell a) = \ell s'(a), \ \forall a   \in R_1.$
\end{enumerate}
\medskip

Then the following properties also are satisfied:
\medskip

\begin{enumerate}\eroman \item $s'(e_\ell) = \ell$ for
all $\ell \in L.$

\pSkip

 \item $s'(ab) \ge s'(a)s'(b)$ for all $a,b \in R.$

\pSkip
 \item
$s'(a) \ge s(a), \ \forall a  \in R.$\end{enumerate}
\medskip

To see this, take $a \in R_k$ and $b \in R_\ell,$ and write $a =
e_k a_1$ and $b = e_\ell b_1$ for $a_1, b_1 \in R_1$. Then
 $$s'(e_\ell) =
s'(e_\ell \rone) = \ell s'(\rone) = \ell\cdot 1 = \ell$$ Taking
$c_1 = a_1b_1 \in R_1,$ we have $$s'(ab) \ds = s'( e_k e_\ell c_1)
\ds = k \ell s'(c_1) = k \ell     s'(a_1)s'(b_1) = s'(a)s'(b),$$
and $s'(a) = s'(e_k a_1 ) = k s'(a_1) \ge k.$
\pSkip

\item The natural injections $R_{\ge 1 } \to R$ and   $\{ \bigcup_\ell {R_\ell: \ell \in \Net} \} \to
R$
  are all examples of layered
homomorphisms.
\pSkip

\item The $L$-truncation map of
\cite[\S3]{IzhakianKnebuschRowen2009Refined} is a layered
homomorphism.
\pSkip

\item Suppose $R$ is a layered pre-\domain0. We adjoin $\infty$ to
$L$, and take $R_\infty$ and $\nu _{\infty, k}$ to be the direct
limit of the $R_k$ and $\nu_{\ell,k}$, and write $\nu$ for the
various $\nu _{\infty, k}.$ An element in $a \in R_1$ is
$\nu$-\textbf{non-cancellative} if $ ab \nucong ac$ for suitable
$b,c,$ where $b\not \nucong c.$ We define the map $\varphi: R \to
R$ which is the identity on $\nu$-cancellative elements but
$\varphi(a) = a^\nu$ for all $\nu$-non-cancellative elements $a
\in R.$ In particular, $\varphi(R)_1$ is comprised precisely of
the $\nu$-cancellative tangible elements.

We claim that $\varphi$ is a homomorphism. If $ab$ is  $\nu$-cancellative this is clear, so we may assume that $a$ is $\nu$-non-cancellative. Then
 $$\varphi(ab) =  (ab)^\nu = a^\nu \varphi(b) = \varphi(a)\varphi(b)    .$$
Certainly $\varphi(a+ b) = \varphi(a)+\varphi(b) $ by bipotence
unless $a \nucong b,$ in which case $$\varphi(a+ b) =
\varphi(a^\nu) = a^\nu = \varphi(a)+\varphi(b)  .$$

Furthermore, $\varphi(R)$ is a layered pre-\domain0 which is
$\nu$-cancellative with respect to $\varphi(R)_1$. Indeed, if $
\varphi(a)\varphi(b) \nucong \varphi(a)\varphi(c)$ with  $
\varphi(a)\in \varphi(R)_1$,
 then $ \varphi(b) \nucong
\varphi(c).$

Note that this is not the same example used in \cite{IKR5}. \end{enumerate}
\end{example}

Our main example for future use is to be given in
Remark~\ref{Fun2}.

\section{The layered categories and the corresponding tropicalization functors}\label{supfun}

Having assembled the basic concepts, we are finally ready for the
layered tropical categories. Our objective in this section is to
introduce the functor that passes from the ``classical algebraic
world'' of integral domains with valuation to the ``layered
world,''
  taking the cue from
\cite[Definition~2.1]{IzhakianRowen2007SuperTropical}, which we
recall and restate more formally.

Here are our first main layered categories, starting with the more
encompassing and proceeding to the specific. In each case the
morphisms are the relevant
 layered homomorphisms.

\begin{defn}\label{somecat} $ $
\begin{enumerate}
 \ddispace \ealph

\item  $\QLayPreD$ is   the  category whose objects are layered
pre-\domains0. \pSkip
%

\item  $\LayDom$ is   the full subcategory of  $\QLayPreD$ whose
objects are layered  \bidomains0. \pSkip

\item  $\ULayDom$ is   the  full subcategory of  $\QLayPreD$ whose
objects are uniform  layered  \bidomains0.

\end{enumerate}
\end{defn}

%
%
\subsection{Identifications of categories of monoids and layered
\predomains0}

\begin{rem}\label{forg2}  We   define the forgetful functor $  \ULayDom\to \SOMon$ given by
sending any uniform $L$-layered \bidomain0 $R: =  \scrR(L,\tG)$ to
$R_1.$
\end{rem}

We want retracts for this forgetful functor. By
 Proposition  \ref{tangen1}, any layered homomorphism corresponds to a homomorphism of
the underlying  monoid of tangible elements, thereby indicating an
identification between categories arising from the construction of
 layered \bidomains0 from  pre-ordered monoids.

\begin{thm}\label{fun1} There is a faithful  \textbf{layering functor}  $$\tFlay : \SOMon \To
\ULayDom,$$ given by sending $\tG$ to $\scrR(L,\tG)$, and sending
the ordered homomorphism $\vrp: \tG \to \tG'$ to the layered
homomorphism $\scrR(L,\tG)\to \scrR(L,\tG')$ induced by $\vrp.$
The functor $\tFlay$ is a left retract of the forgetful functor of
Remark~\ref{forg2}.
\end{thm}
\begin{proof} The image of a cancellative ordered monoid $\tG$ is a layered
\bidomain0, in view of
\cite[Proposition~2.3]{IzhakianKnebuschRowen2009Refined}, and one
sees easily that $\tFlay \vrp $ is a layered morphism since, for
$a \ge_\nu b,$
$$\tFlay \vrp ( \xl{a }{k } +  \xl{b}{\ell })
\nucong \tFlay \vrp ( \xl{a }{k }) \nucong \vrp ( \xl{a }{k
})+\vrp ( \xl{b }{\ell }),$$ and $s'(\tFlay \vrp ( \xl{a }{k } +
\xl{b}{\ell }))\ge k.$

Also, the morphisms match. The functor $\tFlay$ is faithful, since
one recovers the original objects and morphisms by applying the
forgetful functor of Remark~\ref{forg2}.
\end{proof}

More subtly, at times we want to forget the order on our monoids,
 to apply the theory of \cite{CHWW}. Even so, we have a universal construction
 with respect to ``universal characteristic.''

\begin{example}\label{lay0}
 Given a cancellative monoid $\tG$ and a partially  ordered semiring $L$, define the \semiring0
 $U_L(\tG)$ as follows:

Each element of $U_L(\tG)$ is a formal sum of elements of $\tG$,
each supplied with its layer, i.e., has the form
$$ \bigg \{ \sum _{a\in S} \xl{a }{\ell _a }: S \subset \tG\bigg
\}.$$  Addition is given by the
rule
\begin{equation}\label{B1}
\sum _{a\in S} \xl{a }{\ell _a} + \sum _{a\in S'} \xl{a }{\ell _a
'} = \sum _{a\in S\cup S'} \xl{a }{\ell _a + \ell _a'}
.\end{equation} Here we formally define $\ell _a + \ell _a'$ to be
$\ell _a$ (resp.~$\ell _a'$) if $a \notin S'$ (resp.~if $a \notin
S$).

Multiplication is given by

\begin{equation}\label{133}   \xl{a}{k} \cdot  \xl{b}{\ell} =
\xl{(ab)}{k\ell},\end{equation}  extended via distributivity.
\end{example}

We want  $U_L(\tG)$ to be the universal to the forgetful functor
of Corollary~\ref{fun1}. This is ``almost'' true, with a slight
hitch arising from \eqref{B1}.

 \begin{prop}\label{lay01} Given any monoid
   bijection $\vrp: \tM \to \tG$ where  $\tG : = (\tG, \cdot \, , \leq, \one_\tG )$ is a
    totally ordered monoid, viewed as a bipotent \semiring0 as in Proposition~\ref{maxplus1}, there is a
    natural homomorphism $$\htvrp: U_L(\tM) \to \scrR(L,\tG)$$ given by
    $$\htvrp \bigg(\sum_{a \in S} \xl{a }{\ell _a } \bigg) = \sum_{a \in S} \xl{ \vrp (a_1)}{\ell _a }.$$
    The composite \begin{equation*}
    \xymatrix{
  \tM    \ar[r]^{\vrp}&   \tG \ar[r] &\scrR(L,\tG).\\
}
\end{equation*}
   also factors naturally as
\begin{equation*}
    \xymatrix{
  \tM    \ar[r] & (U_L(\tM), \cdot \; )  \ar[r]^{\htvrp} & \scrR(L,\tG).\\
}
\end{equation*}

    In case $L = \{ 1\}$ (the max-plus setting) or $L = \{ 1 , \infty \}$
    (the standard supertropical setting),
    the previous assertion holds more generally for any monoid homomorphism  $\vrp: \tM \to \tG$.
    \end{prop}
\begin{proof} The multiplication rules match, so the verifications
follow formally, cf.~Remark~\ref{Fro0}. The last assertion is true
because in these particular situations the Frobenius property  is
an identity, holding even when $a \nucong  b.$
 \end{proof}

\begin{rem} Since the Frobenius property is an identity, one could
just mod it out from our construction of $ U_L(\tM)$ utilizing
Remark~\ref{Pass00}, and thus get a universal with respect to
satisfying the Frobenius property.
\end{rem}


\subsection{The layered tropicalization functor}\label{RedTro}

Having our categories in place, we can get to the heart of
tropicalization.

\begin{defn}\label{def:Tfunctor} Given a \semiring0 $L$, the $L$-\textbf{tropicalization functor}
$$\TropfunoneL :  \SValmon \To \ULayDom$$ from the category of valued
monoids (with cancellation in the target) to the category of
uniform layered \bidomains0 is defined as follows: $\TropfunoneL:
(\tM,\tG,v) \mapsto \scrR(L,\tG) $ and
 $\TropfunoneL: \phi \mapsto  \al_\phi ,$   where given a morphism $\phi :  (\tM,\tG,v)  \to  (\tM',\tG',v') $ we define
 $\al_\phi : \scrR(L,\tG) \to \scrR(L,\tG') $, by \begin{equation}\label{morph20}
\al_\phi (\xl{a }{\ell }) := \xl{\phi(a)}{\ell }, \qquad a \in
\tG,
\end{equation}
cf. Formula  \eqref{eq:valMonMor}.
\end{defn}
  We also consider $\TropfunoneL$
as acting on individual elements of $\tM$, whereby
\begin{equation}\label{morph201}\TropfunoneL(a) = \xl{v(a)}{1}.\end{equation}
This is indeed a functor, in view of
\cite[Theorem~4.9]{IzhakianKnebuschRowen2009Refined}.

Note that the tropicalization functor $\TropfunoneL$  factors as
$ \SValmon \to \SOMon \to \ULayDom $.

\subsection{More comprehensive layered tropicalization
functors}\label{FulTro}

The basic layered tropicalization functor only recognizes the
image in~ $\tG$, so loses much information about the original
monoid $\tM$. In analogy to \cite{Par}, in order to preserve
information, we can encode extra information, motivated by the
residue field in valuation theory.

\subsubsection{The unit tropicalization
functor}\label{FulTro1}

\begin{defn}\label{units}
Given a monoid $\tM := (\tM, \cdot \, , \one_\tM)$ with
m-valuation $v: \tM \to \tG$, we define its \textbf{unit
submonoid}
$$\tM_{(\one)} := \{a \in \tM: v(a) = \one_\tG\},$$
the submonoid of $\tM$ on which the restriction of $v$ is the
trivial valuation. \end{defn}

When $\tM$ is a group, then $\tM_{(\one)}$ also is a group. In
particular, the category $\Valmon_{(\one)}$ of unit monoids with
m-valuation is a full
subcategory of the category $\Valmon$. 
%
%
%
%
%
%

\begin{example} In Example \ref{Pass111},  $\mathbb C_{(\one)}$ is the complex unit circle.
\end{example}

%
%
%
%
%

The following observation is now clear.

\begin{prop}\label{fullfun}
There is a functor $$\Tropfuntwo :  \SValmon   \To   \ULayDom
\times {\Valmonone} ,$$ given  as follows: $\Tropfuntwo
((\tM,\tG,v)) = (\scrR(L,\tG), \tM_{(\one)}) $ and
 $\Tropfuntwo(\phi) = (\al_\phi , \phi|_{\tM_{(\one)}}),$ where the morphism $\al_\phi :
\scrR(L,\tG) \to \scrR(L',\tG ') $ is
 given by Equation~\eqref{morph20}.
 \end{prop}
\begin{proof} We piece together the two functors.
\end{proof}

This functor could be interpreted as separating the m-valuation
$v$ into two components, corresponding to the value monoid and the
residue domain.   Tropicalization in its original form involved taking the logarithm of the absolute value of $re^{i\theta}$, which is just $\log|r|.$ Thus, the argument  $e^{i\theta}$ is lost, and researchers dealt with that separately. Since these all have absolute value 1, it seems appropriate in the valuation-theory analog to have $\Tropfuntwo$ at our disposal.

A more direct approach in the terminology of
Remark~\ref{terminology0}: Given two Puiseux series $p,q \in
\mathbb K$ with $ \nVal(p) = \nVal(q)$, we see that $\Val(p) =
\Val (q)$ iff $\Val (p q^{-1}) = \Val(1) = \one_K,$ i.e.,
$pq^{-1}-1$ is in the valuation ideal of the valuation $\nVal$.
 Thus, Proposition~\ref{fullfun} gives us a way of
understanding $\Tropfuntwo $ in terms of $\nVal$. Namely, we check
whether two Puiseux series have the same lowest order exponent,
and then can check whether their lowest coefficients are the same
by means of
the residue field.

\begin{rem}\label{cos1}
Suppose $W$ is an arbitrary integral domain with valuation $v: W \setminus \{ \zero_W \} \to
\tG$, with valuation ring $R$ and residue domain $\bar W$.  Take
the unit submonoid
 $W_1$ of $W$, cf.~Definition~\ref{units}. 
Clearly $$W_1 = \{ r \in R: r+ \mfp = 1+\mfp\}.$$

When $W$ is a field, $W_1$ is a multiplicative subgroup of $W$
which could be thought of as the ``first congruence subgroup'' in
valuation theory. Then, for $b\ne 0,$  $aW_1 = b W_1$ iff $v(a) =
v(b)$ and $1 - a b^{-1}\in \mfp,$  which relates to the condition
of the previous paragraph.
 \end{rem}

\subsubsection{The exploded tropicalization
functor}\label{FulTro1}

One could preserve  more information, according to
Parker~\cite{Par}, who introduced ``exploded'' tropical
mathematics, and Payne~\cite{Payne09}. This entails taking the
leading coefficient of Puiseux series.

E.~Sheiner introduced a related structure $\mathcal R(K, \Real)$
on Puiseux series, in which he uses the residue field ~$K$ as the
sorting set.  Define the map $K \to \mathcal R(K, \Real)$ by $p
\mapsto \xl{v(p)}{\alpha}$ where $\alpha$ is the coefficient of
the lowest monomial of the Puiseux series $p$. This map,
generalizing the Kapranov map, keeps track of the ``leading
coefficient'' of the Puiseux series $p$ in terms of when the image
of $p$ has layer 0.

From this perspective, the $\zero_K$ layer represents the ``corner
ghosts.'' Thus, Sheiner has ``exploded'' the notion of valuation,
and it is not difficult to define the ``exploded functor'' and
transfer the statement and proof of Payne~\cite{Payne09} to this
context, to be indicated in \S\ref{Kapr}. Let us describe this
procedure in algebraic terms,  which means working in the
associated graded algebra.

\begin{defn}
Given a valued monoid $v:\tM \to \tG$, and $g \in \tG,$ we
write $ \tM_{\ge g}$ for the $\tM$-module
$\{ a \in \tM: v(a) \ge g\}$, and $\tM_{>g}$ for its submodule $\{ a
\in \tM: v(a)
> g\}.$

When $\tM$ is the multiplicative monoid of an integral domain $W$,
 we can define the \textbf{associated graded algebra}
$$\gr(W) : = \bigoplus _{a\in \tG} W_{\ge g}/W_{>g},$$ where operations
are given by $$(a+ W_{>g})(b+ W_{>h})= ab+ W_{>gh}\qquad (a+ W_{>g})+(b+ W_{>h})= a+b+ W_{>gh}.$$
\end{defn}

\begin{rem} It is well known that the associated graded algebra is an algebra,
with the natural valuation ~$\hat v$ induced by $v$, i.e., $\hat
v(a +  W_{>g}) = v(a).$ When the valuation $v$ is discrete, each
component $W_{\ge g}/W_{>g}$ is (multiplicatively)  isomorphic to
$\bar W$.
\end{rem}

Let us interpret ``explosion'' with respect to Puiseux series. For
any real number $\al$, the component $ K_{\ge \al}/K_{>\al}$ can be identified with $K t^\al,$ which as a module
is isomorphic to $K$, by means of taking the coefficient of
the monomial of lowest order in a Puiseux series.


%
%

%
 \begin{defn}\label{cos2} Notation as in Remark~\ref{cos1},
 define the  \textbf{exploded} layered \domain0 $\mathcal R(\bar W, \tG)$. In other words,
 we sort the elements according to $\bar W$,  with multiplication following the given multiplication in $\tG$
 and addition given by the following rules:

\begin{equation}\label{A14}
  \xl{  x}{\hat a}+   \xl{ y}{\hat b} =\begin{cases}   \xl{  x}{\hat a} & \quad\text{if} \quad x >
y ,\\   \xl{ y}{\hat b}  & \quad\text{if} \quad x <
y, \\
 \xl{x}{\hat a + \hat b}  & \quad\text{if} \quad x =
y.\end{cases}\end{equation}
\end{defn}

\begin{rem}\label{cos30} Note that addition here is the classical addition induced from the integral domain
$W$, so although this structure has a tropical aroma, it does
preserve some of the original algebraic structure of the residue
domain $\bar W$.
\end{rem}

\begin{prop}\label{fullfun1}
There is a functor $$\Tropfunexp : \Valdom   \To   \ULayDom \times
\Ring ,$$ given  as follows: $\Tropfunexp ((W,\tG,v)) =
(\scrR(L,\tG), \bar W) $ and
 $\Tropfunexp(\phi) = (\al_\phi , \overline{\phi}),$ where the morphism $\al_\phi:
\scrR(L,\tG) \to \scrR(L',\tG ') $ is
 given by Equation~\eqref{morph20} and $\overline{\phi}$ is the induced map on the residue domains.
 \end{prop}
\begin{proof} As in Proposition~\ref{fullfun}, we piece together the two functors.
\end{proof}

To preserve even more information, one could sort instead with
$\gr(W)$.

%
%

 \section{The function category}\label{polyfun}

We assume throughout that $R$ is an $L$-layered \domain0. In the
next section we describe layered varieties in terms of corner
roots of ideals of polynomials over $R$. Thus, we need some
preliminaries about the polynomial \semiring0 over a layered
\bidomain0; this is no longer bipotent.

\subsection{The layered function monoid and
\domain0}\label{funmono}

As noted in the introduction, one significant difference between
the tropical theory and ``classical'' algebra is that different
tropical polynomials can agree as functions (whereas for algebras
over an infinite field, any two distinct polynomials are also
distinct as functions). The clearest way of coping with this
phenomenon is to treat polynomials directly as functions from some
subset of $R^{(n)}$ to an extension of $R$, and this enables us to
unify various other constructions related to polynomials.

\begin{definition}\label{ghost0} For any set $\tSS$ and monoid $\tM$, $\Fun (\tSS,\tM)$ denotes
the set of functions from $\tSS$ to~ $\tM$. \end{definition}

\begin{rem}\label{Funobs} $ $\begin{enumerate}\eroman

\item  $\Fun (\tSS,\tM)$  becomes a monoid, under pointwise
 multiplication, i.e.,
\begin{equation}\label{point1}  (fg)(\bfa) = f(\bfa)
g(\bfa), \qquad \forall f,g \in \Fun (\tSS,R), \quad \forall \bfa
\in \tSS.
\end{equation} \pSkip

\item If the monoid $\tM$ is partially ordered, then $\Fun
(\tSS,\tM)$ is also partially ordered, with respect to taking $f
\ge g$ when $f(\bfa) \ge g(\bfa)$ for all $\bfa\in \tSS.$
\end{enumerate}

\item  When moreover  $R$ is a \semiring0, $\Fun (\tSS,R)$ also
becomes a \semiring0, under pointwise addition, i.e.,
\begin{equation}\label{point2} (f+g)(\bfa) = f(\bfa) + g(\bfa),  \quad \forall \bfa
\in \tSS,
\end{equation}
cf.~\cite[Definition 5.1]{IzhakianKnebuschRowen2009Refined}.

\end{rem}

\begin{lem}\label{props} $ $ \begin{enumerate}\eroman

\item If a monoid $\tM$ is cancellative, then  the function monoid
$\Fun (\tSS,\tM)$ is cancellative. \pSkip

\item If a  layered \predomain0 $R$ is cancellative, then  the
function \semiring0 $\Fun (\tSS,R)$ is also a cancellative
 layered pre-\domain0 (but not bipotent!). \pSkip

\item If a \semiring0 $R$ satisfies the Frobenius property
\eqref{eq:Frobenius}, then $\Fun (\tSS,R)$ also satisfies the
Frobenius property.
\end{enumerate}
 \end{lem}
\begin{proof} (i): By pointwise verification.  For cancellation, note
that if $fg = fh$, then $f(\bfa)g(\bfa) = f(\bfa)h(\bfa)$ for all
$\bfa \in \tSS,$ implying $g(\bfa) =
 h(\bfa)$ and thus $g = h$. \pSkip

 (ii): Same verification as in (i). \pSkip

(iii): For  the Frobenius property,
\begin{equation}\label{Frobfun} (f+g)^n (\bfa) = ((f+g)(\bfa))^n =  f(\bfa)^n + g(\bfa)^n, \qquad \forall f,g \in
\Fun (\tSS,R),\end{equation}  for all positive $n \in
\Net$.\end{proof}

 There is a  natural \semiring0 injection $R \rightarrow \Fun
(\tSS,R)$, given by viewing $r\in R$ as the constant function
~$f_r$ given by $f_r(\bfa) = r,$ $\forall \bfa \in \tSS.$ In this
way, we view $R$ as a sub-\semiring0\ of $\Fun (\tSS,R)$. At
first, we take $\tS$ to be $R^{(n)}$. Later  we will take $\tS$ to
be a given ``layered
 variety.''  More generally,
following Payne \cite[\S 2.2]{Payne08}, one could take the set
$\tSS$ to be the lattice of characters of an algebraic torus.

\subsection{Functorial properties of the function monoid and \semiring0}

We categorize the discussion of Section ~\ref{funmono}. First we define the function and polynomial
categories.

%
%

\begin{defn} $\mathcal F := \Fun_{\Mon}(\tSS,\unscr \,  )$ is the functor from
$\Mon $ to $ \Mon$
 given by $\tM \mapsto
\Fun(\tSS,\tM) $ for objects, and such that for any morphism
$\varphi: \tM \to \tM',$ we define $\mathcal F \varphi :
\Fun(\tSS,\tM)\to \Fun(\tSS,\tM')$ to be given by
$$\mathcal F\varphi (f)(\bfa) =  \varphi
(f(\bfa)).$$

The functor $\mathcal F := \Fun_{\SeRo}(\tSS,\unscr \, )
  : \SeRo \to \SeRo$ is  given by $R \mapsto
\Fun(\tSS,R) $ for objects, and again such that for any morphism
$\varphi: R \to R',$ $\mathcal F\varphi : \Fun(\tSS,R)\to
\Fun(\tSS,R')$ is given by
$$\mathcal F \varphi (f)(\bfa) =  \varphi
(f(\bfa)).$$
\end{defn}

\begin{lem}\label{funres}
$\Fun_{\Mon}(\tSS,\unscr \,  )$ and $ \Fun_{ \SeRo }(\tSS,\unscr
\,  )$ are functors. Furthermore, $\Fun_{\Mon}(\tSS,\unscr \,  )$
restricts to a functor from $\SOMon $ to $ \SPOMon$.
 \end{lem}
\begin{proof} The verifications are straightforward, in view of Remark~\ref{Funobs}(ii) and Lemma~\ref{props}.
\end{proof}

\begin{defn} We denote the respective images of $\SOMon$ and $\SeRo $ under the  functors
 $\Fun_{ \Mon}(\tSS,\unscr \,  ) $
 and $ \Fun_{\SeRo }(\tSS,\unscr \,  )$as
$\Fun(\tSS,\SOMon)$ and $\Fun(\tSS,\SeRo)$,  which are respective
subcategories of $\SPOMon$ and $\SeRo$.\end{defn}

Now Proposition~\ref{maxplus1} says:
\begin{prop}\label{maxplus10} There is
 a faithful functor $$\tF_{(\tSS,\SOMon)}: \Fun(\tSS,\SOMon) \To \Fun(\tSS,\SeRo),$$ induced by the
 functor $\tF_{\OMon}$ of Proposition~\ref{maxplus1}, as described in the proof. \end{prop}
\begin{proof} We define $\tF_{(\tSS,\OMon)}(\Fun(\tSS,\tM)) =  \Fun(\tSS,\tM)$ (viewing $\tM$ as a semiring) and,
for any monoid homomorphism $\varphi: M \to M', $
$\tF_{(\tSS,\OMon)}(\varphi):  f \mapsto \varphi \circ f.$ This is
clearly a functor, and is faithful since $\tM$ is embedded into $
\Fun(\tSS,\tM)$.\end{proof}

\begin{prop}\label{funfunc} The
 functors $ \tF_{(\tSS,\OMon)}(\tSS,\unscr \,  ) $
 and $ \Fun_{\SeRo }(\tSS,\unscr \,  )$
 commute with $ \tF_{\OMon}$ of Proposition~\ref{maxplus1}, in
the sense that  $$\tF_{(\tSS,\OMon)} \Fun_{\Mon}(\tSS,\unscr \, )
= \Fun_{ \SeRo }(\tSS,\unscr \,  ) \Fun_{ \OMon} .$$\end{prop}
\begin{proof} Letting $R $ be the \semiring0 of Proposition~\ref{maxplus1},
we have  $ \tF_{(\tSS,\OMon)} \tF_{\Mon}(\tSS,R )(\tM) =
\Fun(\tSS,R) = \Fun_{ \SeRo }(\tSS,R   ) \Fun_{ \OMon} (\tM) .$
\end{proof}


\begin{lem} Construction~\ref{defn50} is functorial, in the sense that
$$\Fun (\tSS,\scrR(L,\tG)) \ds \approx \scrR(\Fun (\tSS,L),\Fun
(\tSS,\tG)).$$
\end{lem}
\begin{proof} Any $f\in \Fun (\tSS,\scrR(L,\tG))$ is given by $f(\bfa) = \xl  b k $ for suitable $k\in L$
and ~$b\in \tG$; we define $f_L$ and $f_\tG$ by $f_L (\bfa) =  k $
and $f_\tG (\bfa) =   b.$ Now $f\mapsto  \xl {f_\tG}{ f_L}$
defines a \semiring0 homomorphism $$\Fun (\tSS,\scrR(L,\tG)) \to
\mathscr R(\Fun (\tSS,L),\Fun (\tSS,\tG)).$$ Conversely, given
$f_L \in \Fun (\tSS,L) $ and $f_\tG \in \Fun (\tSS,\tG) $ we
define $f \in \Fun (\tSS,\scrR(L,\tG))$ by putting $$f(\bfa) = \xl
{f_\tG(\bfa)}{f_L(\bfa)}.$$ One sees that the sorts are preserved.
\end{proof}
%

 \subsection{Sorting the function semiring$^\dagger$}\label{semgho1}

\begin{rem}\label{Fun2} If $\tS' \subseteq \tSS$, there is
a natural \semiring0\ homomorphism $$\Fun (\tSS,R) \to \Fun
(\tS',R)$$ given by $f\mapsto f|{_{\tS'}}.$ In particular, for
$\tSS '= \{ \bfa\},$ we have the evaluation homomorphism at
$\bfa$.
\end{rem}

One main interest in the layered theory is the nature of these
homomorphisms. To understand them, we need to introduce the
appropriate sorting function.

\begin{rem}\label{Funsort} When $L$ is a partially ordered \semiring0, $\Fun (\tSS,L) $
is also a \semiring0 (whose unit element is the constant function
1), which by Remark~\ref{Funobs}(iii) is partially ordered by the
relation:
$$\text{ $f \le_\tSS g \qquad $   {if} $\qquad f(\bfa) \le g(\bfa) \quad$ for all $\bfa \in
\tSS.$}$$ When $L$ is directed from above, this partial order also
is directed from above,  since  $f(\bfa),g(\bfa) \in
 L$ are bounded by $ \max\{f(\bfa) , g(\bfa)\}.$
\end{rem}

If $R$ is  $L$-layered, then $ \Fun (\tSS, R)$ inherits the
layered structure from $R$ pointwise with respect to
  $\Fun (\tSS,L) $, in the following sense taken from \cite[Remark
  5.3]{IzhakianKnebuschRowen2009Refined}:

\begin{defn}\label{def:layMap} The $L$-\textbf{layering map} of a function $f\in   \Fun (\tSS, R)$ is
the map $ \vmap _f: \tSS\to L$ given by $$ \vmap _f(\bfa) :=
s(f(\bfa)), \qquad \bfa \in \tSS.$$ For a set ${ \tI} \subset  \Fun (\tSS, R)$ we define
$$\vmap_{ \tI}(\bfa) := \min \{ \vmap_f(\bfa): f \in { \tI} \}.$$
\end{defn}

 In the layered theory, we only consider
functions that are $\nu$-\textbf{compatible}, in the sense that if
$\bfa \nucong \bfa',$ then $f(\bfa) \nucong f(\bfa')$. 

\begin{example} $ \vmap_{\{ \rone\}}$ is the given sorting map on $R$.
\end{example}

\begin{example}\label{tanZar}$ $  Take $R = \mathcal R(\Net, \Real).$ Assume that $\tSS =
R_1^{(2)}= \Real^{(2)}.$ The examples are written in logarithmic
notation; e.g., $\one := 0$ is the multiplicative unit, and
$2\cdot 3 = 5.$
\begin{enumerate}\eroman

\item Take $f_k = \la_1^k + \la_2 + 0$ for $k \in \Net,$ and $\bfa
= (a_1,a_2) \in \tSS$.

$ \vmap _{f_k}(\bfa) = \begin{cases}   3 & \text{ for } a_1 = a_2
= 0;
\\ 2 & \text{ for } a_1 = 0 >  a_2 \quad \text{ or }\quad a_1 = 0
>  a_2 \quad \text{ or }\quad  a_1 ^k =  a_2 > 0; \\ 1 & \text{
otherwise. }
\end {cases}$ \pSkip

\item Take
  $\tI =\{ f_k : k \in \Net \},$ $\bfa = (a_1,a_2) \in
\tSS$. In view of (i),
$$ \vmap _{\tI}(\bfa) = \begin{cases}   3 & \text{ for } a_1 = a_2
= 0;  \\ 2 & \text{ for } a_1 = 0 >  a_2 \quad \text{ or }\quad
a_1 = 0
>  a_2 ;
\\ 1 & \text{ otherwise. }\end {cases}$$
Thus, the 2-layer is the union of two perpendicular rays.

 \item
Take $\tI = \{ \la _1 +2, \la _1 + 3 \}$. The layering map $
\vmap_{\tI}$ restricted to the tangible elements is
identically~$1,$ the same as that of a tangible constant, although
the ideal generated by $\tI$ does not contain any constants.

Nevertheless, we can distinguish between $\tI$ and tangible
constants, by assuming that $\tSS$ contains elements of $R$ having
layer $
>1$. For example, $\vmap_{\tI}
 (\xl{4}{2}) = 2$ whereas $\vmap_f$ for a tangible constant function $f$ is
 identically~1. \pSkip

\end{enumerate}
\end{example}

As noted in \cite{IzhakianKnebuschRowen2009Refined}, we   layer
the \semiring0 $\Fun (\tSS,R) $ with respect to the sorting
\semiring0 $\Fun (\tSS,L) $, by sending $f\mapsto \vmap_f$.

Given $f,g \in \Fun (\tSS,L) $, write $\tlk = \vmap_f$ and $\tlell
= \vmap _g.$ When $\tlell > \tlk$ we define the transition map
$$\nu _{\tlell,\tlk}: \Fun (\tSS,R)_ {\tlk} \to \Fun (\tSS,R)_
{\tlell} $$ by
$$\nu _{\tlell,\tlk} (f) : \bfa \to \nu_{\tlell
(\bfa), \tlk (\bfa) }(f(\bfa)), \qquad \forall \bfa \in \tSS.$$

 \begin{lem}\label{Layer1} If $R$ is a layered pre-\domain0 with
 partial
 pre-order $\ge_{\nu}$, then we can extend $\nucong$ and $\ge_{\nu}$ respectively to an equivalence and
 a partial
 pre-order on $\Fun (\tSS,R)$ as follows:

 \begin{enumerate} \eroman

    \item $f \nucong g$ iff $f(\bfa) \nucong g(\bfa), \quad \forall \bfa
 \in \tSS;$ \pSkip

 \item  $f \ge_\nu g$ iff $f(\bfa)\ge_\nu   g(\bfa), \quad \forall \bfa
 \in \tSS.$
 \end{enumerate}
 \end{lem}
\begin{proof}  An easy point-by-point verification.
\end{proof}

 We usually start with a given layered \domain0 $R$, and then
 apply Lemma~\ref{Layer1}.
This
rather general framework encompasses some very useful concepts,
including polynomials,  Laurent polynomials, etc. 

\subsection{Polynomials}

We want to understand tropical algebraic geometry in terms of
roots of polynomials. Specifically, we work in the sub-\semiring0
of $\Fun (\tSS, R)_{\ge 1}$ (for  $\tS \subseteq   R^{(n)}$) defined by formulas in the elementary
language under consideration, which we call \textbf{polynomial
functions}. Thus, in the usual language of \semirings0, $R[\Lm]:=
R[\lm_1, \dots, \lm_n]$ denotes the usual polynomials, whose image
in $\Fun (\tSS, R)_{\ge 1}$  we denote as  $\Pol (\tSS,R)$.
 If we adjoin the symbol $^{-1}$
(for multiplicative inverse), then $\Lau (\tSS,R)$ denotes the
image of the Laurent polynomials $R[\Lm, \Lm^{-1}]:= R[\lm_1,
\lm_1^{-1} \dots, \lm_n, \lm_n^{-1}].$ If our language includes
the symbol $ \sqrt[m]{\phantom{w}},$ i.e., if we are working over
a divisible monoid, then we would consider polynomials with
rational powers, which are well-defined in view of Equation \eqref{eq:Frobenius}; although this case is important, we do not treat
it here explicitly because of the extra notation involved.

 Thus,
we are working in the full subcategories $\Pol(\tSS,\SOMon)$ and
$\Lau(\tSS,\SOMon)$ of $\Fun(\tSS,\SOMon)$, and the full
subcategories $\Pol(\tSS,\SeRo)$ and $\Lau(\tSS,\SeRo)$ of $
\Fun(\tSS,\SeRo)$. The functor of Proposition~\ref{maxplus10}
restricts to faithful  functors $\Pol(\tSS,\OMon) \to
\Pol(\tSS,\SeRo)$ and $\Lau(\tSS,\OMon) \to \Lau(\tSS,\SeRo)$.

 The difficulty with treating polynomials (as well as Laurent
polynomials) as functions could be that two polynomial functions
may agree on $R$ but differ on some extension \semiring0 of~$R$.
Fortunately, in \cite[Theorem~5.33 and
Corollary~5.34]{IzhakianKnebuschRowen2009Refined} we saw that
taking $\tlR$ to be the 1-divisible closure of the \bisemifield0
of fractions of~$R$,  if two polynomial functions differ on some
extension of $R$, then they already differ on $\tlR.$ Thus, it
suffices to look at $\Pol (\tSS,\tlR)$ and $\Lau (\tSS,\tlR)$.
Strictly speaking, this was proved only for the specific
construction used in \cite{IzhakianKnebuschRowen2009Refined}, so
to work with  layered \bidomains0 we need to generalize the
construction of 1-divisible closure to $L$-layered \bidomains0.

\begin{example}[1-localization]\label{loc1} If $R$ is an $L$-layered \bidomain0, then
taking any multiplicative submonoid $S $ of $ R_1$, we can form
the localization $S^{-1}R$ as a monoid, and define addition via
$$ \frac au + \frac bv = \frac {av+bu}{uv}$$ for $a,b \in R,$
$u,v \in S$. $S^{-1}R$ becomes an $L$-layered \bidomain0 when we
define $s(\frac au) = s(a).$ There is a natural layered
homomorphism $R \to S^{-1}R$ given by $a \mapsto \frac a \rone ,$
which is injective since $R_1$ is cancellative.

Taking $S = R_1$, we call $S^{-1}R$ the  $L$-layered
\textbf{\bisemifield0 of fractions} of $R$; this construction
shows that any uniform $L$-layered \bidomain0 can be embedded into
a uniform $L$-layered \bisemifield0.
  \end{example}

   \begin{example}[$\nu$-divisible closure]\label{divcl1} We say that an $L$-layered \bidomain0  $R$ is $\nu$-\textbf{divisible} if for each $a\in R$
and $n \in \Net$ there is $b\in R$ such that $b^n \equiv a$ under
the equivalence of Definition~\ref{equivrel}. Note that if $s(a) =
\ell$ then $s(b) = \root n \of \ell.$ This implies that $L$ must
be closed under taking $n$-th roots for each $n$. Assuming that
$L$ is a group satisfying this condition, one can construct the
$\nu$-divisible closure, sketched as follows:

\begin{description}
\item[Step 1] Given $a \in R_\ell,$ adjoin a formal element $b \in
R_{\root n \of \ell}$, and consider all formal
sums\begin{equation}\label{prim2} f(b) :=  \sum_{i}  \al_i
 b^{i}: \quad \al_i \in R.\end{equation}
($  \sum_{i}  \al_i
 b^{i}$ is to be considered as the $n$-th root of  $ \sum_{i}  \al_i^{n}
 a^{i}.$)

  Define
$R_b$ to be the set of all elements of the form \eqref{prim2},
where any  $\a \in R $  is identified with $\al b^0$. 
We can
define the sorting map $s: R_b \to L$ via $$s(f(b)) =\sum _{i}
s(\al_i) \root n \of \ell^i \in L.$$

 We define $\nucong$ on $R_b$ (notation as in \eqref{prim2}) by saying
  $f_b \nucong \sum_{j=0}^{t'}
 \al'_j
 b^{j}$ if  $ \sum _i \al_i^n
a^{i} \nucong \sum _j  {\al'_j}^n
 a^{j}$. In particular, $f_b \nucong c$
for $c\in R$ if  $  \sum _i \al_i^n a^{i}\nucong c^n$. Likewise,
we write $f_b >_\nu f'_b: = \sum_{j=0}^{t'}
 \al'_j
 b^{i}$  if $ \al_i^n
a^{i} > _\nu  {\al'_j}^n
 a^{j}$.

Now we can define addition on $R_b$ so as to be $\nu$-bipotent,
where for   $\nu$-equivalent elements we define $f(b)+g(b)$ to be
their formal sum (combining coefficients of the same powers of
$b$);
 multiplication is then defined in the obvious way, via distributivity over addition.
 Now $R_b$ is an $L$-layered \bidomain0, in view of
 Proposition~\ref{removesort}.

\pSkip

\item[Step 2] Using Step 1 as an inductive step, one can construct
the $\nu$-divisible closure by means of Zorn's Lemma, analogously
to the well-known construction of the algebraic closure,
cf.~\cite[Theorem~4.88]{Row2006}. \pSkip
\end{description}
  \end{example}

 \begin{example}[Completion]\label{compl} One can construct the \textbf{completion} of any $L$-layered
    \domain0
 $R$ as follows: We define $\nu$-\textbf{Cauchy sequences} in $R$ to
be those sequences $(a_i) := \{ a_1, a_2, \dots \}$ which become
Cauchy sequences modulo $\nucong,$ but which satisfy the extra
property that there exists an $m$ (depending on the sequence) for
which $s(a_i) = s(a_{i+1})$, $\forall i \ge m.$ This permits us to
define the \textbf{sort} of the $\nu$-Cauchy sequence to be
$s(a_m).$ Then we define the \textbf{null} $\nu$-\textbf{Cauchy
sequences} in $R$ to be those sequences $(a_i) := \{ a_1, a_2,
\dots \}$ which become null Cauchy sequences modulo $\nucong,$ and
the completion $\htR$ to be the factor group.

We also extend our given pre-order $\nu$ to $\nu$-Cauchy sequences
by saying that $(a_i) \nucong (b_i)$ if $(a_ib_i^{-1})$ is a null
$\nu$-Cauchy sequence, and, for $(a_i) \not \nucong (b_i)$, we say
$(a_i)
>_\nu (b_i)$ when there is $m$ such that $a_i >_\nu b_i$ for all
$i>m$. The completion $\htR$ becomes an $L$-layered \bidomain0
under the natural operations, i.e., componentwise multiplication
of $\nu$-Cauchy sequences, and addition given by the usual rule
that
\begin{equation}\label{144}
(a_i) + (b_i)=\begin{cases}  (a_i)& \quad\text{if}\ (a_i) >_\nu
(b_i),\\ (b_i)& \quad\text{if}\ (a_i) < _\nu (b_i),\\
 \nu_{s(a_i)+s(b_i),s(a_i+b_i)}(a_i+b_i)& \quad\text{if}\ (a_i)\nucong (b_i).\end{cases}\end{equation}

(In the last line, we arranged for the layers to be added when the
$\nu$-Cauchy sequences are $\nu$-equivalent.) It is easy to verify
$\nu$-bipotence for $\htR$.
  \end{example}

These constructions are universal, in the following sense:

\begin{prop}\label{univ} Suppose there is an embedding $\vrp: R \to F'$ of a uniform $L$-layered
\domain0 $R$  into a  $1$-divisible, uniform $L$-layered
\bisemifield0 $F',$ and let $F$ be the $1$-divisible closure of
the \bisemifield0 of fractions of $R$. Then $F'$ is an extension
of $F$. If $F'$ is complete with respect to the $\nu$-pre-order,
then we can take $F'$ to be an extension of the completion of $F$.
\end{prop}
\begin{proof} This is standard, so we just outline the argument.
First we embed the $L$-layered \bisemifield0 of fractions of $R$
into $F'$, by sending $\frac {b} {a_1} \to \frac
{\vrp(b)}{\vrp(a_1)}.$ This map is 1:1, since if $\frac {b} {a_1}
= \frac {d} {c_1},$ then $c_1 b = a_1d,$ implying $ \vrp(c_1 b) =
\vrp(a_1d),$ and thus $\frac {\vrp(b)}{\vrp(a_1)} = \frac
{\vrp(d)}{\vrp(c_1)}.$ Thus, we may assume that $R$ is an
$L$-layered \bisemifield0. Now we define the map $F \to F'$ by
sending $\root m \of a  \to \root m \of \vrp(a)$, for each $a \in
F_1.$ This is easily checked to be a well-defined, 1:1 layered
homomorphism.

In case $F'$ is complete, then we can  embed  the completion of
$F$  into $F'$. (The completion of a $1$-divisible \bisemifield0
is $1$-divisible, since taking roots of a $\nu$-Cauchy sequence in
$F_1$ yields a $\nu$-Cauchy sequence.)
\end{proof}

\begin{cor} Any $L$-layered homomorphism $\Phi: R\to R'$ of uniform $L$-layered \bidomains0 extends
uniquely to an $L$-layered homomorphism $\tlPhi: \tlR \to \tlR'$.
\end{cor}
\begin{proof} Each of the constructions of \bisemifield0
of fractions, 1-divisible closure, and completion are universal,
so applying them in turn yields us a unique ordered monoid
homomorphism from the tangible component~$R_1$ to $R'_1$, which
readily extends (uniquely) to all of $R$ since $e_\ell \mapsto
e'_\ell$.\end{proof}

\subsection{Roots and layered varieties}


 In order to understand affine layered geometry, we need to know
more about the affine layered algebraic sets. We fix the sorting
\semiring0 $L$ for convenience, although one could also let $L$
vary, and think of $\tS$ as a subset of $R^{(n)}$, where  $R$ is
an $L$-layered \bidomain0. Standard tropical geometry can be
recaptured
by taking $\tS \subseteq   R_1^{(n)}$.

\begin{defn} Suppose  $f,g\in \Pol(\tSS,R) $. We say that  $f$ \textbf{dominates} $g$ at $\bfa \in \tS$ if $f(\bfa ) \nuge g(\bfa)$.
Write $f = \sum_i f_i$ as a sum of monomials. The \textbf{dominant
part} $f _\bfa$ of $f$ at $\bfa$ is the sum of all those $f_i$
dominating~$f$ at $\bfa,$ i.e., for which  $f_i (\bfa ) \nucong
f(\bfa ).$ We write $f|_X$ for the restriction of $f$ to a
nonempty subset $X$ of $\tSS$.
\end{defn}

\begin{defn}\label{def:corn}
An element $\bfa \in \tSS$ is a \textbf{corner root} of $f =
\sum_i f_i$ if $f(\bfa)$ is $s(f_i(\bfa))$-ghost for every
monomial $f_i$ of $f$. (Thus, $ f _\bfa$ contains at least two
$f_i$.) The (affine) \textbf{corner locus}
 of
$f\in \Pol(\tSS,R)$
 with respect to  the set $\tSS$ is
$$\tZ_\corn (f;\tSS) : = \{ \bfa \in \tSS \ds : \bfa
\text{ is a corner root of }f \}.$$ We write $\tZ_\corn (f)$ for
$\tZ_\corn (f;R^{(n)})$.

An element $\bfa \in \tSS$ is a \textbf{cluster root} of $f$ if
$f(\bfa) = f_i(\bfa)$ is $1$-ghost for some monomial $f_i$ of $ f
_\bfa$.  (Thus, $ f _\bfa$ is comprised of a single monomial
$f_i$.) The \textbf{combined ghost  locus}
 with respect to  the set $\tSS$ is
$$\tZ_\tot (f;\tSS) : = \{ \bfa \in \tSS \ds : \bfa
\text{ is a cluster or corner root}\}.$$
%

%
%

 The \textbf{(affine) corner
algebraic set} and the \textbf{(affine) algebraic set} of a subset
$A \subseteq \Pol(\tSS,R)$ with respect to the set $\tSS$,
  are  respectively
$$\tZ_\corn (A;\tSS) : = \bigcap _{f\in A} \tZ_\corn(f;\tSS), \qquad  \tZ_\tot (A;\tSS) : = \bigcap _{f\in A} \tZ_\tot(f;\tSS).$$
%
  When $\tSS$ is unambiguous (usually $R^{(n)}$), we
write $\tZ_\corn(A)$ and $\tZ_\tot(A)$  for $\tZ_\corn(A;\tSS)$
and $\tZ_\tot(A;\tSS)$ respectively.
\end{defn}
%

\begin{example} Here are basic examples of  affine layered
algebraic sets.
\begin{enumerate} \eroman
    \item We view $R^{(n)}$ as $ \tZ_\tot(\emptyset)$. Note that also $R^{(n)}= \tZ_\tot(\{\bfa\})$
for any ``ghost'' constant $\bfa \in R_{>1}.$  \pSkip

    \item The empty set is an algebraic set: $\emptyset = \tZ(\{\bfa\}) $
    for any $\bfa \in R_1$. \pSkip

    \item A single point $\bfa = (a_1, \dots,a_n) \in \tSS$,  where $\tS \subset R^{(n)}_1$,  is a  corner algebraic
    set: $$\bfa = \tZ( \{ \lm_1 + a_1, \dots, \lm_n + a_n\}).$$

    \item
 The familiar \textbf{tropical line} in the affine plane is $\tZ_\corn
 (f;\tSS)$ where $f$ is linear of the form $$\al \la_1 + \bt \la _2 +
 \gm,$$ with $\al, \bt, \gm \in R_1$, and $\tS = R_1^{(2)}$.
 On the other hand, for $I$ as in Example~\ref{tanZar}, $\tZ_\corn
 (I ;\tSS)$ restricted to $R_1^{(2)}$ is the union of two perpendicular rays, and does
not satisfy the celebrated ``balancing condition'' of tropical
geometry.
\end{enumerate}
\end{example}

 \begin{example}\label{lay02} A more sophisticated
example: Whereas in the standard supertropical theory we have 
$$(x+y+z)(xy+xz+yz) = (x+y)(x+z)(y+z),$$
they differ in the layered theory, since $xyz$ has layer 3 in the left side but only layer 2 in the right side.
Thus the layered theory permits greater refinement in reducing tropical varieties.
\end{example}

%

\section{The tropicalization functor on polynomials and their roots}\label{Kapr}

The tropicalization map $\TropfunoneL$ of \S\ref{RedTro},
Equation~\eqref{morph201}, extends readily to polynomials, i.e.,
to the functor $\hTropfunoneL :  \Pol (\tSS,\SValmon) \to
\Pol(\tSS,\LayDom),$ where we define
$$\hTropfunoneL \bigg(\sum_\bfi a_\bfi \la_1 ^{i_1}\cdots \la_n ^{i_n}\bigg) =
 \sum_\bfi  \TropfunoneL   (a_\bfi) \la_1 ^{i_1}\cdots \la_n
^{i_n}   =  \sum_\bfi \xl{v(a_\bfi)}{1} \la_1 ^{i_1}\cdots \la_n
^{i_n},
 $$ for $\bfi = (i_1, \dots, i_n),$ (and analogously for morphisms).

 If $\bfa \in F^{(n)}$ is a root of $f
\in F[\Lambda],$ then clearly $v(\bfa)$ is a corner root of
$\hTropfunoneL(f)$. We are interested in the opposite direction.
One of the key results of tropical mathematics is Kapranov's
theorem ~\cite{IMS}, which    says that for any
polynomial $f(\la_1, \dots, \la_n)$, any corner root of the
tropicalization of $f$ has a pre-image which is a root of~$f$.
This assertion also works for finite sets of polynomials, and
thus for ideals, in view of \cite{Payne09}. Our objective in this
section is to understand this result in terms of the appropriate
layered categories.
%

\begin{rem}
Let $\mcA: = F[\lm_1,\dots, \lm_n]$. Then the Puiseux series
valuation $\Val$ extends naturally to a map  $\Val: \mcA \to
R[\lm_1,\dots, \lm_n]$, where each $\la_i$ is fixed. If $\tI$ is
an ideal of $\mcA$, then $\Phi(\tI)$ is an ideal of $\Phi(\mcA),$
so this ``tropicalization'' process sends ideals of algebras to
\semiring0 ideals, and transfers many properties from the
``classical algebraic'' world to the ``tropical'' world. One
property which it does not preserve is
 generation of ideals. For example, two
different polynomials $f,g$ of $\tI$ might have the same leading
monomial and the same tropicalization, and then $\Phi(f + (-g))$
cannot be described in terms of $\Phi(f)$ and $\Phi(g) = \Phi(f).$
One can bypass this particular difficulty by using Gr\"{o}bner
bases (since they are comprised of polynomials of different lowest
orders), but the necessity of choosing the ``right'' generators
raises serious issues in tropical geometry.  Fortunately, this
concern is not critical in the current paper, since we do not
require generators for studying the relevant categories.
\end{rem}

\begin{rem}\label{Puis} We start with a triple $(F,\tG,v)$, where $F$ for example may
be the algebra of Puiseux series over $\C$, $\tG= (\Real,+)$, and
$v: F \to \tG$ the valuation $\nVal$. Any point $(\a_1, \dots,
\a_n)\in F^{(n)}$ can be considered as a valuation $ \hat v$
extending $v$, where $\hat v (\la_i) = \a_i.$ This can be extended
to the group $G $ generated by the $\la_i$ and $\la _i^{-1}$ in
the ring of Laurent series over $F$. But if $\mfp$ is a prime
ideal of $F[\Lambda]:= F[\la_1, \dots, \la _n]$, then the natural
image of $G$ in the field of fractions $K$ of  $F[\Lambda]/\mfp$
is a group $\bar G$, and Bieri-Groves~\cite{BG} describe the
possible extensions of $v$ to $F[\bar G].$ Namely, $\bar G$, being
a finitely generated Abelian group, can be written as the direct
sum of a free Abelian group of some rank $m$ and a torsion group
$T$. We let $F[\bar G]$ denote the $F$-subalgebra of $K$ generated
by $\bar G$. After extending the valuation $v$ to the free Abelian
group, one sees by an exercise of Bourbaki~\cite{B} that further
extensions to $F[\bar G]$ correspond to corner roots of the
polynomials of $\mfp.$
\end{rem}

This is explained in the proof of \cite[Theorem~A]{BG}, and can be
explained tropically in terms of the  proof of Bourbaki's exercise:

If $f(a) = 0$ then two of the monomials of $\TropfunoneL(f)$ must
be equal and dominant when evaluated at~ $a$, say $\a a^i = \bt
a^j$, so one can extend $v$ to a valuation $\hat v$ on $F[a]$ given by $\hat v(a) = \frac {v(\bt \a^{-1})}{i-j}.$

This discussion could be formulated in the language of
\cite{IzhakianKnebuschRowen2009Valuation},
 \cite{IKR3}, \cite{IKR5}, as  elaborated in \cite[Remark~6.6]{IKR6}.

\begin{defn} As in Remark~\ref{cos1}, suppose $F$ is an arbitrary field with valuation $v: F \to \tG$, having valuation ring $R$ and
associated graded algebra~$\gr (F)$.  For any $f   \in F[\la]$, we
define $\bar f$ to be its natural image in $\gr(F)[\la].$ For $\mfa
\triangleleft F[\Lambda],$ we define the \textbf{exploded
tropicalization} $\bar \mfa$ of $\mfa$ to be $$\{ \bar f : f \in \mfa\}.$$

An element $\bfa := ( a_1, \dots,  a_n)$ of $ \gr(F)$ is a
\textbf{graded  root} of a polynomial $\bar f \in \gr (F)[\la]$ if
$$\bar f(a_1+  F_{>s(a_1)}, \dots, a_n  +F_{>s(a_n)}) = 0$$ in $\gr (F)$.
(Intuitively, $s(f(a_1, \dots, a_n))$ is larger than expected.)

\end{defn}
%
%

We take $F $ to be a Henselian field with respect to a
valuation~$v$ whose residue field is algebraically closed. For
example, we could take $F= \mathbb K$, the field of Puiseux
series over $\mathbb C.$
 We have two areas
of interest when studying Puiseux series -- the \semifield0 (which
corresponds to the value group) and the residue field, which is a
copy of $\mathbb C.$  We can combine these using the `exploded'
structure of Definition~\ref{cos2}. Given a polynomial $f\in
\gr(F)[\Lambda],$ we define its \textbf{corner exploded roots} to
be
$$\big \{ \text{Graded roots }\bfa = (\hat a_1, \dots, \hat a_n) \in F_1^{(n)}
  \text{ of }f: s(f(\hat a_1, \dots, \hat a_n)) = 0 \big \},$$
 cf.~Remark~\ref{cos30}. The \textbf{corner exploded variety} of an ideal $\hat \mfa$ of $\gr(F)$ is
 the set of common  corner exploded roots of the
 polynomials of $\hat \mfa$.

 The standard
valuation-theoretic proofs of  Kapranov's theorem show that any
corner root $x$ of $\TropfunoneL(f)$ is the tropicalization of a
point in  the variety  $Z$ defined by $f$. In other words, $x$
lifts to an exploded root of $f$. Payne's theorem~\cite{Payne09}
can be stated as follows:

Suppose $X$ is an affine variety defined by a proper ideal $\mfa$ of
$F[\Lambda],$ and $\bfa$ is an graded root of the exploded
tropicalization $\bar \mfa$ of~$\mfa$. Then the preimage of any point
defined by $\bar \mfa$, if nonempty, is Zariski dense in $X$. This is
the algebraic essence of Parker's `exploded' approach.
%
%

\section {The category of  affine layered geometry}

Our goal in this section is to connect   affine layered geometry
 to a category which can be studied by means of standard
algebraic techniques. This ties in with the algebraic categories
of the previous sections, by means of the coordinate \semiring0,
which is to be studied more thoroughly in a subsequent paper.
Throughout, let $F$ denote a layered \bisemifield0.

\subsection{The Zariski topology}

 We
want to mimic the classical Zariski theory as far as we can,
starting with our layered \bisemifield0 $F$ and describing a
topology on $\tSS,$ a given subset of $F^{(n)}.$

Actually, there are several natural   topologies on $\tSS$.

\begin{defn} Suppose $f = \sum_i f_i$ is written as a sum of
monomials in $\Pol(\tSS,R)$. The ${f_i}$-\textbf{component} of $f$
is
$$D_{f_i}(\tSS) := \{\bfa \in \tSS: f_i(\bfa) = f(\bfa)\}.$$
\end{defn}

Any root of $f$ in a component $D_{f_i}$ must be a cluster root.


 \begin{rem}\label{Zar1} In \cite[Definition~6.5]{IzhakianKnebuschRowen2009Refined},
 we defined the $L$-layered \textbf{component topology} to
 have as its sub-base the components of polynomials of $\Pol(\tSS,R)$.
Note that different components of a polynomial are disjoint, so
open sets here are not necessarily dense in the component
topology. Thus, although it provides useful information, the
component topology is too fine to permit us to develop tropical
algebraic geometry along classical lines.
 We rectify the
situation by defining the \textbf{principal corner open sets} to
be
$$\tD(f;\tS) = \tS \setminus \tZ_{\corn}(f) = \bigcup_{i \in I}
 D_{f_\bfi},$$
where   $f = \sum_{\bfi \in I} f_\bfi$ is written as a sum of
monomials in $\Pol(\tSS,R)$. Put another way, $$\text{ $\tD(f;\tS)
= \{ \bfa \in \tSS: s(f(\bfa)) = s(f_i(\bfa))$ for some monomial
$f_i$ of $f\},$}$$

 The principal corner open sets form a base for a
topology
 on $\tSS$, which we call the \textbf{($L$-layered) corner Zariski topology},
 whose closed sets are
 affine corner algebraic sets.

Analogously, one could respectively take cluster roots and  use
$\tZ_\tot (f;\tSS)$ in place of corner roots and  $\tZ_{\corn}
(f;\tSS)$ to define the \textbf{combined Zariski
 topology}, whose closed sets are the algebraic sets. This is a somewhat
 finer topology, but the corner Zariski topology provides a
 closer analog to the usual notions of tropical geometry, so we
 will use that.
\end{rem}

\begin{lem}\label{top0} The intersection of two principal corner open sets contains
 a nonempty principal corner open set.
\end{lem}
\begin{proof} If $ f(\bfa) = f_ i (\bfa)$ on $D_{f,i}$ and  $ g(\bfa) = g_ j
(\bfa)$ on $D_{g,j}$, then clearly $ fg(\bfa) = f_ i (\bfa)
g_j(\bfa) $  on $D_{f,i} \cap D_{g,j}$ and nowhere else.
\end{proof}

\begin{prop} All open sets in the ($L$-layered) corner Zariski topology
%
 are  dense. \end{prop} \begin{proof}
Immediate from the lemma.
 \end{proof}
%
%



\subsection{The coordinate \semiring0}

We can return to algebra via the coordinate \semiring0, just as in
classical algebraic geometry.

\begin{defn}\label{coord2} The \textbf{coordinate \semiring0}
of an affine  layered algebraic set $X \subseteq \tS$, denoted
$F[X],$ is the natural image of the \semiring0 $\Pol(X,F)$.   The
\textbf{Laurent coordinate \semiring0} $F(X)$ is the natural image
of $\Lau(X,F)$; its elements are called the \textbf{regular
functions} of the algebraic set.
\end{defn}

  $\Pol(X,F)$ can be identified with classes of
polynomials over $X$ whose representatives are polynomials having
no inessential monomials, cf.~ \cite[Definition
5.5]{IzhakianKnebuschRowen2009Refined}. We say that a function  in
$F[X]$ (resp.~$F(X)$) is \textbf{tangible} if it can be written as
a tangible polynomial (resp.~Laurent polynomial), i.e., having
coefficients only in $F_1$.

\begin{remark}\label{rmk:XinYcord}
When $X \subset Y$ we have a natural map  $ F[Y]\to F[X]$ obtained
by restricting the domain of the function from $Y$ to $X$.
\end{remark}

 Coordinate \semirings0
correspond naturally to congruences on $\Pol(\tS,F)$ in the
following manner.

\begin{defn}\label{Zarcor01} A nonempty subset $X \subseteq \tSS$ defines the
\textbf{congruence of $X$ on $\Fun(\tS,F)$}, denoted $\Cong _X$,
whose underlying equivalence $\cngA_X$ is given by $$\text{$f
\cngA _X g$ \dss{iff} $f(\bfa) = g(\bfa)$ for every $\bfa \in
X.$}$$ Conversely, given a congruence $\Cong$ on $\Fun(\tS,F)$,
define the \textbf{variety of the congruence}
$$V(\Cong) := \{ \bfa \in \tS : f(\bfa) = g(\bfa) , \  \forall
(f,g) \in \Cong\} \subseteq \tS.$$\end{defn}

It is readily checked that
$$\Cong_{X \cup Y} = \Cong_X \cap \Cong_Y, \qquad \text{for any } X, Y \subset
\tS.$$

\subsection{Zariski
correspondences}

We have various correspondences between varieties and the
algebraic structure.

\subsubsection{The Zariski
correspondence with ideals}

Inspired by the layered Nullstellensatz given in
\cite[Theorem~6.14]{IzhakianKnebuschRowen2009Refined}, the naive
approach would be to define the \textbf{corner ideal} $\tI_\corn
(\tSS)$ of a set $\tSS$ to be $$\{ f \in \Pol (\tSS,F): \bfa
\text{ is a corner root of }f, \quad \forall \bfa \in \tSS \},$$
and $\tI_\tot (\tSS)$   to be $$\{ f \in \Pol (\tSS,F): \bfa
\text{ is in the combined ghost locus of }f, \quad \forall \bfa
\in \tSS \}.$$

This approach misses the mark, somewhat. On the one hand, different congruences can define
the same ideal which is the pre-image of 0.
 On the other hand, there
are ``too many'' ideals, in the sense that not every ideal defines
a variety, and the correct algebraic approach is to utilize
congruences rather than ideals.
 Furthermore, we need somehow to filter out those
varieties obtained by degenerate intersections of hypersurfaces;
this is treated in a later paper under preparation.

\subsubsection{The Zariski correspondence with congruences}
As just noted, it makes more sense to deal with congruences
instead of ideals. We have the usual straightforward but important
inverse Zariski correspondence:

\begin{prop}\label{Zarcor1}
If $\Cong_1 \supseteq \Cong_2,$ then $V(\Cong_1) \subseteq
V(\Cong_2).$ Conversely, if $Y \supseteq X,$ then  $\Cong_Y
\subseteq  \Cong_X$. Consequently, $$\text{$V(\Cong_{V(\Cong)}) =
V(\Cong)$ \ and \ $\Cong_{V(\Cong_X)} = \Cong_X.$}$$ It follows
that there is a 1:1 correspondence between congruences of
varieties and varieties of congruences, given by $X \mapsto
\Cong_X$ and $\Cong \mapsto V(\Cong).$ Furthermore, the coordinate
\semiring0 satisfies $$F[V(\Cong)] \ds \cong \Pol(\tS,F)/\Cong.$$
\end{prop}
\begin{proof} The inverse correspondence is immediate, and the next
assertion is immediate. The 1:1 correspondence is then formal. To
see the last assertion, note that two polynomials in $f,g$ are
identified in $F[V(\Cong)]$ iff they agree on $ V(\Cong)$, which
by  definition is the point set on which every pair $(f,g) \in
\Cong$ agree; namely, $f$ and $g$ are identified in
$\Pol(\tS,F)/\Cong.$
\end{proof}

By the proposition, one sees that  for any nonempty subset $X
\subseteq \tSS$ we have  \begin{equation}\label{eq:1} F[X] \ds
\cong   \Pol(\tSS, F) / \Cong_X.
\end{equation}

\begin{defn}\label{10.34} A  \textbf{morphism} of affine layered algebraic sets  $\Phi :
X \to Y$ is a continuous function that preserves (i.e., pulls
back) regular functions,  in the sense that if $U$ is an open
subset of $Y$ and $\psi \in F(U),$ then $\psi \circ \Phi \in F (
\Phi ^{-1} (U))$.
\end{defn}

$\LAff$  is the category whose objects are the affine layered
algebraic sets $ X \subset F ^{(n)}$ and whose morphisms  $\Phi: X
\to Y$ are morphisms of layered affine algebraic sets .

\begin{prop}\label{10.35}

  Any morphism $\Phi : X \to Y$ of affine layered algebraic set  gives rise
to a natural algebra homomorphism  $\Phi^*: F (\Phi(U)) \to F(U)$,
by $\psi \mapsto \Phi^*(\psi),$
 where    $\Phi^*(\psi)(\bfa) =
 \psi(\Phi(\bfa)),$ for every  $\bfa \in U$.
\end{prop}
\begin{proof}  $ \Phi^*(\psi+\vrp) = \Phi ^* (\psi) +
\Phi ^* (\vrp)$ and $\Phi^* (\psi \vrp) =
 \Phi^* (\psi) \Phi^* (\vrp).$
\end{proof}

We conclude by introducing the functor linking the algebraic and
geometric (affine) categories.

\begin{defn}
$\tF_{\LCoord}$ is the contravariant functor from  $\LAff$ to
$\ULayDom$ given by sending an affine layered algebraic set $X$ to
its coordinate \semiring0 \ $F[X]$, and any morphism $\Phi : X \to
Y$ of affine layered algebraic sets to the layered \semiring0
homomorphism $\Phi^*: F[Y] \to F[X],$ i.e.,  $f \mapsto f_\Phi$
where $f_\Phi(\bfa) =
 f(\Phi(\bfa))$.
\end{defn}

Many subtleties lie behind this definition; for example, which
affine layered varieties correspond to the coordinate \semirings0
of tropical varieties satisfying the balancing condition? This
question is to be treated in a subsequent paper.

\end{document}